\documentclass[reqno]{amsart}
\usepackage{amssymb,amsmath,amsthm,amstext,amsfonts}
\usepackage[dvips]{graphicx}
\usepackage{psfrag}
\usepackage[ansinew]{inputenc} 

\renewcommand\thetable{\thesection.\@arabic\c@table}

\theoremstyle{plain}
\newtheorem{maintheorem}{Theorem}

\newtheorem{maincorollary}[maintheorem]{Corollary}

\newtheorem{theorem}{Theorem }[section]
\newtheorem{proposition}[theorem]{Proposition}
\newtheorem{lemma}[theorem]{Lemma}
\newtheorem{corollary}[theorem]{Corollary}

\theoremstyle{definition} \theoremstyle{remark}
\newtheorem{remark}[theorem]{Remark}
\newtheorem{example}[theorem]{Example}
\newtheorem{definition}[theorem]{Definition}

\newcommand{\field}[1]{\mathbb{#1}}
\newcommand{\real}{\field{R}}

\newcommand{\torus}{\field{T}}

\newcommand{\al} {\alpha}       
        
\newcommand{\ga} {\gamma}    
\newcommand{\de} {\delta}       \newcommand{\De}{\Delta}

\newcommand{\vep}{\varepsilon}

\newcommand{\la} {\lambda}      \newcommand{\La}{\Lambda}

\newcommand{\si} {\sigma}

\newcommand{\om} {\omega}       \newcommand{\Om}{\Omega}

\newcommand{\Z}{\mathbb{Z}}
\newcommand{\N}{\mathbb{N}}
\newcommand{\R}{\mathbb{R}}
\newcommand{\supp}{\operatorname{supp}}
\newcommand{\diam}{\operatorname{diam}}

\newcommand{\topp}{\operatorname{top}}
\newcommand{\Leb}{\operatorname{Leb}}

\newcommand{\ti}{\tilde }

\newcommand{\Ptop}{P_{\topp}}
\newcommand{\htop}{h_{\topp}}
\newcommand{\Wloc}{W_{\text{loc}}}
\newcommand{\hWloc}{\hat W_{\text{loc}}}
\newcommand{\un}{\underbar}
\newcommand{\cI}{{\mathcal I}}

\newcommand{\cR}{\mathcal{R}}
\newcommand{\cP}{\mathcal{P}}
\newcommand{\cL}{\mathcal{L}}

\newcommand{\cM}{\mathcal{M}}
\newcommand{\cB}{\mathcal{B}}

\newcommand{\cG}{\mathcal{G}}

\newcommand{\cQ}{\mathcal{Q}}
\newcommand{\cF}{\mathcal{F}}
\newcommand{\cT}{\mathcal{T}}
\newcommand{\cW}{\mathcal{W}}
\newcommand{\cU}{\mathcal{U}}
\newcommand{\cS}{\mathcal{S}}

\newcommand{\cA}{\mathcal{A}}
\newcommand{\cZ}{\mathcal{Z}}

\begin{document}

\title{Existence, uniqueness and stability of equilibrium states for
       non-uniformly expanding maps}

\author{Paulo Varandas and Marcelo Viana}

\address{IMPA, Est. D. Castorina 110 \\ 22460-320 Rio de Janeiro, RJ, Brazil}
\email{varandas@impa.br, viana@impa.br}

\thanks{P.V. was partially supported by FCT-Portugal through the grant
        SFRH/BD/11424/2002, Funda\c c\~ao Calouste Gulbenkian, and CNPq.
        M.V. was partially supported by CNPq, FAPERJ, and PRONEX-Dynamical Systems.}
\date{\today}

\begin{abstract}
We prove existence of finitely many ergodic equilibrium states for a
large class of non-uniformly expanding local homeomorphisms on
compact manifolds and H\"older continuous potentials with not very
large oscillation. No Markov structure is assumed. If the
transformation is topologically mixing there is a unique equilibrium
state, it is exact and satisfies a non-uniform Gibbs property. Under
mild additional assumptions we also prove that the equilibrium
states vary continuously with the dynamics and the potentials
(statistical stability) and are also stable under stochastic
perturbations of the transformation.
\end{abstract}

\maketitle

\section{Introduction}\label{s.introduction}

The theory of equilibrium states of smooth dynamical systems was
initiated by the pioneer works of Sinai, Ruelle,
Bowen~\cite{Si72,BR75,Bo75,Ru76b}. For uniformly hyperbolic
diffeomorphisms and flows they proved that equilibrium states exist
and are unique for every H\"older continuous potential, restricted
to every basic piece of the non-wandering set. The basic strategy to
prove this remarkable fact was to (semi)conjugate the dynamics to a
subshift of finite type, via a Markov partition.

Several important difficulties arise when trying to extend this
theory beyond the uniformly hyperbolic setting and, despite
substantial progress by several authors, a global picture is still
far from complete. For one thing, existence of generating Markov
partitions is known only in a few cases and, often, such partitions
can not be finite. Moreover, equilibrium states may actually fail to
exist if the system exhibits critical points or singularities (see
Buzzi~\cite{Buz01}).

A natural starting point is to try and develop the theory first for
smooth systems which are hyperbolic in the non-uniform sense of
Pesin theory, that is, whose Lyapunov exponents are non-zero
``almost everywhere''. This was advocated by Alves, Bonatti,
Viana~\cite{ABV00}, who assume non-uniform hyperbolicity at
\emph{Lebesgue} almost every point and deduce existence and
finiteness of physical (Sinai-Ruelle-Bowen) measures. In this
setting, physical measures are absolutely continuous with respect to
Lebesgue measure along expanding directions.

It is not immediately clear how this kind of hypothesis may be
useful for the more general goal we are addressing, since one
expects most equilibrium states to be singular with respect to
Lebesgue measure. Nevertheless, in a series of recent works,
Oliveira, Viana~\cite{Ol03,OV06,OV07} managed to push this idea
ahead and prove existence and uniqueness of equilibrium states for a
fairly large class of smooth transformations on compact manifolds,
inspired by \cite{ABV00}. Roughly speaking, they assumed that the
transformation is expanding on most of the phase space, possibly
with some relatively mild contracting behavior on the complement.
Moreover, the potential should be H\"older continuous and its
oscillation $\sup\phi-\inf\phi$ not too big. On the other hand, they
need a number of additional conditions on the transformation, most
notably the existence of (non-generating) Markov partitions, that do
not seem natural.

Important contributions to the theory of equilibrium states outside
the uniformly hyperbolic setting have been made by several other
authors: Denker, Keller, Nitecki, Przytycki,
Urba{\'n}ski~\cite{DKU90, DU91a,DU91b,DNU95,DPU96,Ur98}, Bruin,
Keller, Todd~\cite{BrK98,BT06,BT07a}, and Pesin, Senti,
Zhang~\cite{PS05,PS06,PSZ07}, for one-dimensional maps, real and
complex. Wang, Young~\cite{WY01} for H\'enon-like maps. Buzzi,
Maume, Paccaut, Sarig,~\cite{Bu99,BPS01,BM02,BS03} for piecewise
expanding maps in higher dimensions. Buzzi,
Sarig~\cite{BS03,Sa99,Sa01,Sa03,Yu03} for countable Markov shifts.
Denker, Urba{\'n}ski~\cite{DU91d,DU91f,DU92b} and
Yuri~\cite{Yu99,Yu00,Yu03} for maps with indifferent periodic
points. Leplaideur, Rios~\cite{LR06,LR2} for horsehoes with
tangencies at the boundary of hyperbolic systems. This list is
certainly not complete. Some results, including ~\cite{BR06} and
\cite{OV06} are specific for measures of maximal entropy. An
important notion of entropy-expansiveness was introduced by
Buzzi~\cite{Bu00b}, which influenced \cite{Bu05,OV06} among other
papers.

In this paper we carry out the program set by Alves, Bonatti, Viana
towards a theory of equilibrium states for the class of
non-uniformly expanding maps originally proposed in
\cite[Appendix]{ABV00}. We improve upon previous results of
\cite{OV07} in a number of ways. For one thing, we completely remove
the need for a Markov partition (generating or not). In fact, one of
the technical novelties with respect to previous recent works in
this area is that we prove, in an abstract way inspired by
Ledrappier~\cite{Le84a}, that every equilibrium state must be
absolutely continuous with respect to a certain conformal measure.
When the map is topologically mixing, the equilibrium state is
unique, and a non-lacunary Gibbs measure. In this regard let us
mention that Pinheiro~\cite{Pi08} has recently announced an inducing
scheme for constructing (countable) Markov partitions for a class of
non-invertible transformations closely related to ours. Another
improvement is that our results are stated for local homeomorphisms
on compact metric spaces, rather than local diffeomorphisms on
compact manifolds (compare \cite[Remark~2.6]{OV07}). In addition, we
also prove stability of the equilibrium states under random noise
(stochastic stability) and continuity under variations of the
dynamics (statistical stability).

Our basic strategy to prove these results goes as follows. First we
construct an expanding conformal measure $\nu$ as a special
eigenmeasure of the dual of the Ruelle-Perron-Frobenius operator.
Then we show that every accumulation point $\mu$ of the Cesaro sum
of the push-forwards $f^n_*\nu$ is an invariant probability measure
that is absolutely continuous with respect to $\nu$ with density
bounded away from infinity, and that there are finitely many
distinct such ergodic measures. In addition, we prove that these
absolutely continuous invariant measures are equilibrium states, and
that any equilibrium state is necessarily an expanding measure.
Finally, we establish an abstract version of Ledrappier's theorem
\cite{Le84a} and characterize equilibrium states as invariant
measures absolutely continuous with respect to $\nu$.

This paper is organized as follows. The precise statement of our
results is given in Section~\ref{Statement of results}. We included
in Section~\ref{Preliminary results} preparatory material that will
be necessary for the proofs. Following the approach described above,
we construct an expanding conformal measure and prove that there are
finitely many invariant and ergodic measures absolutely continuous
with respect it through Sections~\ref{Conformal Measure} and
~\ref{Absolutely Continuous Measures}. In Section~\ref{Proof Thm 1}
we prove Theorems~\ref{Thm. Equilibrium States} and ~\ref{Thm.
Equilibrium States2}. Finally, in Section~\ref{Proof of Theorems 3,
4 and 5} we prove the stochastic and statistical stability results
stated in Theorems~\ref{Thm. Statistical Stability} and ~\ref{Thm.
Stochastic Stability2}.

\medskip

\textbf{Acknowledgements:} This paper is an outgrowth of the first
author's PhD thesis at IMPA. We are grateful to V. Pinheiro, V.
Araújo and K. Oliveira for very useful conversations.

\section{Statement of the results}\label{Statement of results}

\subsection{Hypotheses}

We say that $X$ is a \emph{Besicovitch metric space} if it is a
metric space where the Besicovitch covering lemma (see e.g.
\cite{Gu75}) holds. These metric spaces are characterized in
\cite{Fed69} and include e.g. any subsets of Euclidean metric spaces
and manifolds.

We consider $M \subset N$ to be a compact Besicovitch metric space
of dimension $m$ with distance $d$. Let $f:M \to N$ be a \emph{local
homeomorphism} and assume that there exists a bounded function
$x\mapsto L(x)$ such that, for every $x\in M$ there is a
neighborhood $U_x$ of $x$ so that $f_x : U_x \to f(U_x)$ is
invertible and
$$
d(f_x^{-1}(y),f_x^{-1}(z))
    \leq L(x) \;d(y,z), \quad \forall y,z\in f(U_x).
$$
Assume also that every point has finitely many preimages and that
the level sets for the degree $\{x : \#\{f^{-1}(x)\}=k\}$ are
closed. Given $x\in M$ set $\deg_x(f)=\# f^{-1}(x)$. Define
$h_n(f)=\min_{x\in M} \deg_x(f^n)$ for $n\ge 1$, and consider the
limit
$$
h(f)=\liminf_{n\to\infty} \frac1n \log h_n(f).
$$
It is clear that
$$
\log\Big( \max_{x\in M} \#\{f^{-1}(x)\}\Big)
    \geq h(f) \geq
\log\Big( \min_{x\in M} \#\{f^{-1}(x)\}\Big).
$$
If $M$ is connected, every point has the same number $\deg(f)$ of
preimages by $f$ and $h(f)=\log \deg(f)$ is the topological entropy
of $f$ (see Lemma~\ref{Gibbs imply equilibrium} below). The limit
above also exists e.g. when the dynamics is (semi)conjugated to a
subshift of finite type. By definition, there exists $N \ge 1$ such
that $\deg_x(f^n) \geq e^{h(f) n}$ for every $x\in M$ and every
$n\ge 1$. Up to consider the iterate $f^N$ instead of $f$ we will
assume that every point in $M$ has at least $e^{h(f)}$ preimages by
$f$.

For all our results we assume that $f$ and $\phi$ satisfy conditions
(H1), (H2), and (P) stated in what follows. Assume that that there
exist constants $\si>1$ and $L>0$, and an open region $\cA\subset M$
such that \vspace{.1cm}
\begin{itemize}
\item[(H1)] $L(x)\leq L$ for every $x \in \cA$ and
$L(x)\leq \sigma^{-1}$ for all $x\in M\backslash \cA$, and $L$ is
close to $1$: the precise conditions are given in \eqref{eq.
relation expansion} and \eqref{eq. relation potential} below.
\item[(H2)] There exists $k_0 \geq
1$ and a covering $\cP=\{P_1,\dots, P_{k_0}\}$ of $M$ by domains of
injectivity for $f$ such that $\cA$ can be covered by $q<e^{h(f)}$
elements of $\cP$. \vspace{.1cm}
\end{itemize}
The first condition means that we allow expanding and contracting
behavior to coexist in $M$: $f$ is uniformly expanding outside $\cA$
and not too contracting inside $\cA$. The second one requires
essentially that in average every point has at least one preimage in
the expanding region. The interesting part of the dynamics is given
by the restriction of $f$ to the compact metric space
$$
K= \bigcap_{n\geq 0} f^{-n}(M),
$$
that can be connected or totally disconnected. We give examples
below where $K$ is a manifold and where it is a Cantor set. In
\cite[Remark~2.6]{OV07} the authors pointed out that their results
could hold in more general metric spaces and for non-smooth maps.

In addition we assume that $\phi:M \to \mathbb R$ is H\"older
continuous and that its variation is not too big. More precisely,
assume that:
\begin{itemize}
\item[(P)]$\sup\phi - \inf\phi <h(f)-\log q$. \vspace{.1cm}
\end{itemize}
Notice this is an open condition on the potential, relative to the
uniform norm, and it is satisfied by constant functions. It can be
weakened somewhat. For one thing, all we need for our estimates is
the supremum of $\phi$ over the union of the elements of $\cP$ that
intersect $\cA$. With some extra effort (replacing the $q$ elements
of $\cP$ that intersect $\cA$  by the same number of smaller
domains), one may even consider the supremum over $\cA$, that is,
$\sup\phi\mid_\cA - \inf\phi <h(f)-\log q$. However, we do not use
nor prove this fact here.

Let us comment on this hypothesis. A related condition,
$\Ptop(f,\phi)> \sup\phi$, was introduced by Denker,
Urba\'nski~\cite{DU91a} in the context of rational maps on the
sphere. Another related condition, $P(f,\phi,\partial\cZ) <
P(f,\phi)$, is used by Buzzi, Paccaut, Schmitt~\cite{BPS01}, in the
context of piecewise expanding multidimensional maps, to control the
map's behavior at the boundary $\partial\cZ$ of the domains of
smoothness: without such a control, equilibrium states may fail to
exist~\cite{Buz01}. Condition (P) seems to play a similar role in
our setting.

\subsection{Examples}\label{subsec.examples}

Here we give several examples and comment on the role of the
hypotheses (H1), (H2) and (P), specially in connection with the
supports of the measures we construct, the existence and finitude of
equilibrium states.

\begin{example}\label{ex.saddle}
Let $f_0:\torus^d\to\torus^d$ be a linear expanding map. Fix some
covering $\cP$ for $f_0$ and some $P_1\in\cP$ containing a fixed (or
periodic) point $p$. Then deform $f_0$ on a small neighborhood of
$p$ inside $P_1$ by a pitchfork bifurcation in such a way that $p$
becomes a saddle for the perturbed local homeomorphism $f$. By
construction, $f$ coincides with $f_0$ in the complement of $P_1$,
where uniform expansion holds. Observe that we may take the
deformation in such a way that $f$ is never too contracting in
$P_1$, which guarantees that (H1) holds, and that $f$ is still
topologically mixing. Condition (P) is clearly satisfied by
$\phi\equiv 0$. Hence, Theorems~\ref{Thm. Equilibrium States} and
~\ref{Thm. Equilibrium States2} imply that there exists a unique
measure of maximal entropy, it is supported in the whole manifold
$\torus^d$ and it is a non-lacunary Gibbs measure.
\end{example}

Now, we give an example where the union of the supports of the
equilibrium states does not coincide with the whole manifold.

\begin{example}\label{Hopf bifurcation}
Let $f_0$ be an expanding map in $\mathbb T^2$ and assume that $f_0$
has a periodic point $p$ with two complex conjugate eigenvalues
$\ti\si e^{i\varpi}$, with $\ti\si>3$ and $k\varpi \not\in 2\pi \Z$
for every $1\leq k\leq 4$. It is possible to perturb $f_0$ through
an Hopf bifurcation at $p$ to obtain a local homeomorphism $f$,
$C^5$-close to $f_0$ and such that $p$ becomes a periodic attractor
for $f$ (see e.g. \cite{HV05} for details). Moreover, if the
perturbation is small then (H1) and (H2) hold for $f$. Thus, there
are finitely many ergodic measures of maximal entropy for $f$. Since
these measures are expanding their support do not intersect the
basin of attraction the periodic attractor $p$.
\end{example}

An interesting question concerns the restrictions on $f$ imposed by
(P). For instance, if $\phi=-\log|\det Df|$ satisfies (P) then there
can be no periodic attractors. In fact, the expanding conformal
measure $\nu$ coincides with the Lebesgue measure which is an
expanding measure and positive on open sets. An example where the
potential $\phi=-\log|\det Df|$ satisfies (P) is given by
Example~\ref{ex.saddle} above, since condition (P) can be rewritten
as
\begin{equation}\label{eq.P}
 \frac{\sup_{x\in \torus^2} |\det Df(x)|}{\inf_{x\in \torus^2} |\det Df(x)|} < \deg(f),
\end{equation}
and clearly satisfied if the perturbation is small enough.

\smallskip

The next example shows that some control on the potential $\phi$ is
needed to have uniqueness of the equilibrium state: in absence of
the hypothesis (P), uniqueness may fail even if we assume (H1) and
(H2).

\begin{example}\emph{(Manneville-Pomeau map)}
If $\al \in (0,1)$, let $f:[0,1]\to [0,1]$ be the local
homeomorphism given by
\begin{equation*}\label{eq. Manneville-Pomeau}
f_\al(x)= \left\{
\begin{array}{cl}
x(1+2^{\alpha} x^{\alpha}) & \mbox{if}\; 0 \leq x \leq \frac{1}{2}  \\
2x-1 & \mbox{if}\; \frac{1}{2} < x \leq 1.
\end{array}
\right.
\end{equation*}
Observe that conditions (H1) and (H2) are satisfied. It is well
known that $f$ has a finite invariant probability measure $\mu$
absolutely continuous with respect to Lebesgue. Using Pesin formula
and Ruelle inequality, it is not hard to check that both $\mu$ and
the Dirac measure $\de_0$ at the fixed point $0$ are equilibrium
states for the potential $\phi = -\log|\det Df|$. Thus,
\emph{uniqueness fails} in this topologically mixing context. For
the sake of completeness, let us mention that in this example $f$ is
not a local homeomorphism, but one can modify it to a local
homeomorphisms in $S^1=[0,1]/\sim$ by
$$
f_\al(x)
  = \left\{ \begin{array}{cl}
           x(1+2^{\alpha} x^{\alpha}) & \mbox{if}\; 0 \leq x \leq \frac{1}{2}  \\
           x-2^\al (1-x)^{1+\al} & \mbox{if}\; \frac{1}{2} < x \leq 1,
             \end{array}
    \right.
$$
where $\sim$ means that the extremal points in the interval are
identified. Note that the potential $\phi$ is not (H\"older)
continuous.
\end{example}

The previous phenomenon concerning the lack of uniqueness of
equilibrium states can appear near the boundary of the class of maps
and potentials satisfying (H1) and (H2) and (P).

\begin{example} Let $f_\al$ be the map given by the previous example
and let $(\phi_\beta)_{\beta>0}$ be the family of H\"older
continuous potentials given by $\phi_{\beta}=
-\log(\det|Df|+\beta)$. On the one hand, observe that $\phi_{\beta}$
converge to $\phi=-\log (|\det Df|)$ as $\beta \to 0$. On the other
hand, similarly to \eqref{eq.P}, one can write condition (P) as
$$
\frac{\beta+2+\al}{\beta+1}< 2,
    \quad\text{or simply}\quad
\beta>\al.
$$
For every $\al>0$, since $f_\al$ is topologically mixing and
satisfies (H1),(H2) and $\phi_{2\al}$ satisfies (P) for every
$\al>0$ there is a unique equilibrium state $\mu_\al$ for $f_\al$
with respect to $\phi_{2\al}$. Moreover, $\phi_{2\al}$ approaches
$\phi$, which seems to indicate that the condition (P) on the
potential should be close to optimal in order to get uniqueness of
equilibrium states.

Since $\htop(f)=\log 2$, condition (P) can be rewritten also as
$\sup\phi-\inf\phi<\htop(f)$. In \cite[Proposition~2]{BT07a}, the
authors proved that for every H\"older continuous potential that
does not satisfy (P) has no equilibrium state obtained from some
'natural' inducing schemes.
\end{example}

The next example illustrates that our results also apply when the
set $K$ is totally disconnected.

\begin{example}\label{ex.Cantor}
Let $f:[0,1]\to \R$ be the unimodal map $f(x)=-8x(x-1)(x+1/8)$.
Since the critical point is outside of the unit interval $[0,1]$,
$K=\cap_n f^{-n}([0,1)]$ is clearly a Cantor set. Although the
existence of a critical point, the restriction of $f$ to the
intervals in $f^{-1}[0,1])$ is a local homeomorphism. It is not hard
to check that (H1) and (H2) hold in this setting and that $f\mid K$
is topologically mixing. As a consequence of the results below we
show that there is a unique measure of maximal entropy for $f$,
whose support is $K$.
\end{example}

\subsection{Existence of equilibrium states}\label{existence eq.
states}

We say that $f$ is \emph{topologically mixing} if, for each open set
$U$ there is a positive integer $N$ so that $f^N(U)=M$. Let $\cB$
denote the Borel $\sigma$-algebra of $M$. An $f$-invariant
probability measure $\eta$ is \emph{exact} if the $\si$-algebra
$\cB_\infty = \cap_{n\geq 0} f^{-n}\cB$ is $\eta$-trivial, meaning
that it contains only zero and full $\eta$-measure sets. Given a
continuous map $f:M\to M$ and a potential $\phi:M \to \mathbb R$,
the variational principle for the pressure asserts that
\begin{equation*}
\label{variational principle} \Ptop(f,\phi)=\sup \left\{
h_\mu(f)+\int \phi \;d\mu : \mu \;\text{is}\; f\text{-invariant}
\right\}
\end{equation*}
where $\Ptop(f,\phi)$ denotes the topological pressure of $f$ with
respect to $\phi$ and $h_\mu(f)$ denotes the metric entropy. An
\textit{equilibrium state} for $f$ with respect to $\phi$ is an
invariant measure that attains the supremum in the right hand side
above.

\begin{maintheorem}\label{Thm. Equilibrium States}
Let $f:M \to M$ be a local homeomorphism with Lipschitz continuous
inverse and $\phi:M \to \R$ a H\"older continuous potential
satisfying (H1), (H2), and (P). Then, there is a finite number of
ergodic equilibrium states $\mu_1, \mu_2, \dots, \mu_k$ for $f$ with
respect to $\phi$ such that any equilibrium state $\mu$ is a convex
linear combination of $\mu_1, \mu_2, \dots, \mu_k$. In addition, if
the map $f$ is topologically mixing then the equilibrium state is
unique and exact.
\end{maintheorem}

Our strategy for the construction of equilibrium states is, first to
construct a certain conformal measure $\nu$ which is expanding and a
non-lacunary Gibbs measure. Then we construct the equilibrium
states, which are absolutely continuous with respect to this
reference measure $\nu$. Both steps explore a weak hyperbolicity
property of the system. In what follows we give precise definitions
of the notions involved.

A probability measure $\nu$, not necessarily invariant, is
\emph{conformal} if there exists some function $\psi:M\to \R$ such
that $$\nu(f(A))=\int_A e^{-\psi} d\nu$$ for every measurable set
$A$ such that $f \mid A$ is injective.

Let $S_n \phi=\sum_{j=0}^{n-1} \phi \circ f^j$ denote the $n$th
Birkhoff sum of a function $\phi$. The \emph{dynamical ball} of
center $x\in M$, radius $\de>0$, and length $n \geq 1$ is defined by
$$
B(x,n,\delta)=\{y\in M : d(f^j(y),f^j(x)) \leq \delta, \;\forall \,
0\leq j \leq n\}.
$$
An integer sequence $(n_k)_{k\geq1}$ is \emph{non-lacunary} if it is
increasing and $n_{k+1}/n_k\to 1$ when $k\to\infty$.

\begin{definition} A probability measure $\nu$ is a
\emph{non-lacunary Gibbs measure} if there exist uniform constants
$K>0$, $P\in \mathbb R$ and $\de>0$ so that, for $\nu$-almost every
$x \in M$ there exists some non-lacunary sequence $(n_k)_{k\geq 1}$
such that
\begin{equation*}\label{eq. Gibbs at hyperbolic times}
K^{-1} \leq \frac{\nu(B(x,n_k,\delta))}
                 {\exp(-P\,n_k + S_{n_k}\phi(y))}\leq K
\end{equation*}
for every $y\in B(x,n_k,\delta)$ and every $k\ge 1$.
\end{definition}

The weak hyperbolicity property of $f$ is expressed through the
notion of hyperbolic times, which was introduced in
\cite{Al00,ABV00} for differentiable transformations. We say that
$n$ is a \emph{$c$-hyperbolic time} for $x \in M$ if
\begin{equation}\label{eq. c-hyperbolic times}
\prod_{j=n-k}^{n-1} L(f^j(x)) < e^{-ck}
 \quad \text{for every} \; 1\leq k\leq n.
\end{equation}
Often we just call them hyperbolic times, since the constant $c$
will be fixed, as in \eqref{eq. relation expansion}. We denote by
$H$ the set of points $x \in M$ with infinitely many hyperbolic
times and by $H_j$ the set of points having $j\ge 1$ as hyperbolic
time. A probability measure $\nu$, not necessarily invariant, is
\emph{expanding} if $\nu(H)=1$.

The \emph{basin of attraction} of an $f$-invariant probability
measure $\mu$ is the set $B(\mu)$ of points $x \in M$ such that
$$
\frac{1}{n} \sum_{j=0}^{n-1} \de_{f^j(x)} \ \ \text{converges weakly
to $\mu$ when $n\to\infty$.}
$$

\begin{maintheorem}\label{Thm. Equilibrium States2}
Let $f:M \to M$ be a local homeomorphism and $\phi:M \to \R$ be a
H\"older continuous potential satisfying (H1), (H2), and (P). Let
$\mu_1, \mu_2 ,\dots ,\mu_k$ be the ergodic equilibrium states of
$f$ for $\phi$. Then every $\mu_i$ is absolutely continuous with
respect to some conformal, expanding, non-lacunary Gibbs measure
$\nu$. The union of all basins of attraction $B(\mu_i)$ contains
$\nu$-almost every point $x \in M$. If, in addition, $f$ is
topologically mixing then the unique absolutely continuous invariant
measure $\mu$ is a non-lacunary Gibbs measure.
\end{maintheorem}

As a byproduct of the previous results we can obtain the existence
of equilibrium states for \emph{continuous} potentials satisfying
(P). Without some extra condition no uniqueness of equilibrium
states is expected to hold even if $f$ is topologically mixing.

\begin{maincorollary}\label{Cor.ContinuousPotentials}
Let $f:M \to M$ be a local homeomorphism satisfying (H1) and (H2).
If $\phi:M \to \R$ is a continuous potential satisfying (P) then
there exists an equilibrium state for $f$ with respect to $\phi$.
Moreover, there is a residual set $\mathcal R$ of potentials in
$C(M)$ that satisfy (P) such that there is unique equilibrium state
for $f$ with respect to $\phi$.
\end{maincorollary}

\subsection{Stability of equilibrium states}\label{stability eq. states}

Let $\cF$ be a family of local homeomorphisms with Lipschitz inverse
and $\mathcal W$ be some family of continuous potentials $\phi$. A
pair $(f,\phi) \in \cF \times \cW$ is \emph{statistically stable}
(relative to $\cF\times\cW$) if, for any sequences $f_n \in \cF$
converging to $f$ in the uniform topology, with $L_n$ converging to
a $L$ in the uniform topology, and $\phi_n \in \cW$ converging to
$\phi$ in the uniform topology, and for any choice of an equilibrium
state $\mu_n$ of $f_n$ for $\phi_n$, every weak$^*$ accumulation
point of the sequence $(\mu_n)_{n\geq 1}$ is an equilibrium state of
$f$ for $\phi$. In particular, when the equilibrium state is unique,
statistical stability means that it depends continuously on the data
$(f,\phi)$.

\begin{maintheorem}\label{Thm. Statistical Stability}
Suppose every $(f,\phi)\in\cF\times\cW$ satisfies (H1), (H2), and
(P), with uniform constants (including the H\"older constants of
$\phi$). Assume that the topological pressure $\Ptop(f,\phi)$ varies
continuously in the parameters $(f,\phi) \in \cF \times \cW $. Then
every pair $(f,\phi) \in \cF \times \cW $ is statistically stable
relative to $\cF\times\cW$.
\end{maintheorem}

The assumption on continuous variation of the topological pressure
might hold in great generality in this setting. See the comment at
the end of Subsection~\ref{subsec.statistical} for a discussion.

Now let $\cF$ be a family of local homeomorphisms satisfying (H1)
and (H2) with uniform constants. A \emph{random perturbation} of
$f\in\cF$ is a family $\theta_\vep$, $0< \vep \le 1$ of probability
measures in $\cF$ such that there exists a family $V_\vep(f)$, $0<
\vep \le 1$ of neighborhoods of $f$, depending monotonically on
$\vep$ and satisfying
$$
\supp\theta_\vep \subset V_\vep(f) \quad\text{and}\quad
\bigcap_{0<\vep\le 1} V_\vep(f) =\{f\}.
$$
Consider the skew product map
\[
\begin{array}{rcl}
F: \cF^\N \times M  &  \to & \cF \times M \\
   (\un f, x) & \mapsto & (\si(\un f),f_1(x))
\end{array}
\]
where $\un f=(f_1, f_2, \ldots)$ and $\si: \cF^\N \to \cF^\N$ is the
shift to the left. For each $\vep>0$, a measure $\mu^\vep$ on $M$ is
\emph{stationary} (respectively, \emph{ergodic}) for the random
perturbation if the measure $\theta_\vep^\N \times \mu^\vep$ on
$\cF^\N \times M$ is invariant (respectively, ergodic) for $F$.

We assume the random-perturbation to be \emph{non-degenerate},
meaning that, for every $\vep>0$, the push-forward of the measure
$\theta_\vep$ under any map
$$
\cF \ni g \mapsto g(x)
$$
is absolutely continuous with respect to some probability measure
$\nu$, with density uniformly (on $x$) bounded from above, and its
support contains a ball around $f(x)$ with radius uniformly (on $x$)
bounded from below. The first condition implies that any stationary
measure is absolutely continuous with respect to $\nu$. In
Theorem~\ref{Thm. Stochastic Stability} we shall use also the second
condition to conclude that, assuming $\nu$ is expanding and
conformal, for any $\vep>0$ there exists a finite number of ergodic
stationary measures $\mu^\vep_1$, $\mu^\vep_2$, $\ldots$,
$\mu^\vep_l$. We say that $f$ is \emph{stochastically stable} under
random perturbation if every accumulation point, as $\vep\to 0$, of
stationary measures $(\mu^\vep)_{\vep>0}$ absolutely continuous with
respect to $\nu$ is a convex combination of the ergodic equilibrium
states $\mu_1$, $\mu_2$, $\ldots$, $\mu_k$ of $f$ for $\phi$.

A \emph{Jacobian} of $f$ with respect to a probability measure
$\eta$ is a measurable function $J_\eta f$ such that
\begin{equation}\label{eq. Jacobian}
\eta(f(A))=\int_A J_\eta f \,d\eta
\end{equation}
for every measurable set $A$ (in some full measure subset) such that
$f\mid A$ is injective. A Jacobian may fail to exist, in general,
and it is essentially unique when it exists. If $f$ is at most
countable-to-one and the measure $\eta$ is invariant, then Jacobians
do exist (see \cite{Pa69}).

\begin{maintheorem}\label{Thm. Stochastic Stability2}
Let $(\theta_\vep)_\vep$ be a non-degenerate random perturbation of
$f \in \cF$ and $\nu$ be the reference measure in Theorem~\ref{Thm.
Equilibrium States2}. Assume $\nu$ admits a Jacobian for every
$g\in\cF$, and the Jacobian varies continuously with $g$ in the
uniform norm. Then $f$ is stochastically stable under the random
perturbation $(\theta_\vep)_{\vep}$.
\end{maintheorem}

The conditions on the Jacobian are automatically satisfied in some
interesting cases, for instance when $\nu$ is the Riemannian volume
or $f$ is an expanding map. This is usually associated to the
potential $\phi=-\log|\det(Df)|$. Example~\ref{ex.saddle} describes
a situation where this potential satisfies the condition (P).


\section{Preliminary results}\label{Preliminary results}

Here, we give a few preparatory results needed for the proof of the
main results. The content of this section may be omitted in a first
reading and the reader may choose to return here only when
necessary.

\subsection{Combinatorics of orbits}

Since the region $\cA$ is contained in $q$ elements of the partition
$\cP$ we can assume without any loss of generality that $\cA$ is
contained in the first $q$ elements of $\cP$. Given $\ga\in(0,1)$
and $n\ge 1$, let us consider the set $I(\ga,n)$ of all itineraries
$(i_0,\dots,i_{n-1}) \in \{1, \dots, k_0\}^n$ such that $\# \{0\le j
\le n-1: i_j\le q\} > \gamma n$. Then let
\begin{equation}\label{eq.c}
c_\ga = \limsup_{n\to\infty} \frac{1}{n} \log \#{I(\ga,n)}.
\end{equation}

\begin{lemma}\cite[Lemma~3.1]{OV07}\label{l.combinatorio}
Given $\vep>0$ there exists $\ga_0 \in (0,1)$ such that
$c_{\ga}<\log q +\vep$ for every $\ga \in (\ga_0,1)$.
\end{lemma}

We are in a position to state our precise condition on the constant
$L$ in assumption (H1) and the constant $c$ in the definition of
hyperbolic time. By (P), we may find $\vep_0>0$ small such that
$\sup \phi - \inf \phi +\vep_0 < h(f) - \log q$. By
Lemma~\ref{l.combinatorio}, we may find $\gamma<1$ such that
$c_\gamma<\log q + \vep_0/4$. Assume $L$ is close enough to $1$ and
$c$ is close enough to zero so that
\begin{equation}\label{eq. relation expansion}
\sigma^{-(1-\gamma)} L^\gamma<e^{-2c} < 1
\end{equation}
and
\begin{equation}\label{eq. relation potential}
\sup \phi -\inf\phi < h(f)-\log q -\vep_0-m \log L
\end{equation}

\subsection{Hyperbolic times}\label{Hyperbolic times}

The next lemma, whose proof is based on a lemma due to Pliss (see
e.g. \cite{Man87}), asserts that, for points satisfying a certain
condition of asymptotic expansion, there are infinitely many
hyperbolic times: even more, the set of hyperbolic times has
positive density at infinity.

\begin{lemma}\label{positive density}
Let $x\in M$ and $n \geq 1$ be such that
$$
\frac{1}{n} \sum_{j=1}^{n} \log L(f^j(x)) \leq -2c<0.
$$
Then, there is $\theta>0$, depending only on $f$ and $c$, and a
sequence of hyperbolic times $1 \leq n_1(x) < n_2(x) < \dots< n_l(x)
\leq n$ for $x$, with $l \geq \theta n$ .
\end{lemma}

\begin{proof}
Analogous to Corollary 3.2 of \cite{ABV00}.
\end{proof}

\begin{corollary}\label{c.definite fraction}
Let $\eta$ be a probability measure relative to which
$$
 \limsup_{n\to \infty} \frac{1}{n} \sum_{j=1}^{n} \log L(f^j(x)) \leq -2c<0
$$
holds almost everywhere. If $A$ is a positive measure set then
$$
\liminf_{n \to \infty} \frac{1}{n} \sum_{j=0}^{n-1} \frac{\eta(A\cap
H_j)}{\eta(A)}\geq \frac{\theta}{2}.
$$
\end{corollary}

\begin{proof}
By Lemma~\ref{positive density}, for $\eta$-almost every point $x\in
M$ there is $N(x)\in \N$ so that $n^{-1}\sum_{j=0}^{n-1}
\chi_{H_j}(x) \geq \theta$ for every $n\geq N(x)$. Fix an integer
$N\geq 1$ and choose $\tilde A\subset A$ so that $\eta(\tilde A)\geq
\eta(A)/2$ and $N(x)\geq N$ for every $x \in \ti A$. If we integrate
the expression above with respect to $\eta$ on $A$ we obtain that
$$
\frac{1}{n} \sum_{j=0}^{n-1} \eta(H_j\cap A) \geq \theta \eta(\ti A)
\geq \frac{\theta}{2}\eta(A)
$$
for every integer $n$ larger than $N$, completing the proof of the
lemma.
\end{proof}

\begin{lemma}\label{delta}
There exists $\delta=\delta(c,f)>0$ such that, whenever $n$ is a
hyperbolic time for a point $x$, the dynamical ball
$V_n(x)=B(x,n,\delta)$ is mapped homeomorphically by $f^n$ onto the
ball $B(f^n(x),\delta)$, with
$$
d(f^{n-j}(y),f^{n-j}(z)) \leq e^{-\frac{c}{2} j} d(f^n(y),f^n(z))
$$
for every $1\leq j \leq n$ and every $y,z \in V_n(x)$.
\end{lemma}

\begin{proof}
Analogous to the proof of \cite[Lemma 2.7]{ABV00}, just replacing
$\log \|Df(\cdot)^{-1}\|$ by $\log L(\cdot)$, and using the
definition of hyperbolic time and the Lipschitz property of the
inverse branches of $f$.
\end{proof}

If $n$ is a hyperbolic time for a point $x \in M$, the neighborhood
$V_n(x)$ given by the lemma above is called \textit{hyperbolic
pre-ball}. As a consequence of the previous lemma we obtain the
following property of bounded distortion on pre-balls.

\begin{corollary}\label{lem. bounded distortion}
Assume $J_\eta f=e^\psi$ for some H\"older continuous function
$\psi$. There exist a constant $K_0>0$ so that, if $n$ is a
hyperbolic time for $x$ then
$$
K_0^{-1} \leq \frac{J_\eta f^n(y)}{J_\eta f^n(z)} \leq K_0
$$
for every $y,z \in V_n(x)$.
\end{corollary}
\begin{proof}
Let $n$ a hyperbolic time for a point $x$ in $M$ and $(C,\alpha)$ be
the H\"older constants of $\psi$. Using Lemma~\ref{delta} it is not
hard to see that
$$
|S_n \psi(y)-S_n \psi(z)|
 \leq C \sum_{j=0}^{+ \infty} e^{-c\al/2 j} d(f^n(x),f^n(y))^{\al}
 \leq C \de^\alpha \sum_{j=0}^{+ \infty} e^{-c\al j/2}.
$$
for any given $y, z \in V_n(x)$. Choosing $K_0$ as the exponential
of this last term and noting $J_\eta f^n$ is the exponential of
$S_n\psi$, the result follows immediately.
\end{proof}

\subsection{Non-lacunary sequences}\label{Non-lacunary sequences}

The set $H$ of points with infinitely many hyperbolic times plays a
central role in our strategy. We are going to see that for such a
point the sequence of hyperbolic times has some special properties.
The first one is described in the following remark:

\begin{remark}\label{r.concatenation}
If $n$ is a hyperbolic time for $x$ then, clearly, $n-s$ is a
hyperbolic time for $f^s(x)$, for any $1 \le s < n$. The following
converse is a simple consequence of ~\eqref{eq. c-hyperbolic times}:
if $k<n$ is a hyperbolic time for $x$ and there exists $1\le s\le k$
such that $n-s$ is a hyperbolic time for $f^s(x)$ then $n$ is a
hyperbolic time for $x$. Thus, if $n_j(x)$, $j\ge 1$ denotes the
sequence of values of $n$ for which $x$ belongs to $H_n$ then, for
every $j$ and $l$
$$
n_j(x) + n_l(f^{n_j(x)}(x)) = n_{j+l}(x)
$$
We will refer to this property as \emph{concatenation} of hyperbolic
times. Moreover, if $n$ is a hyperbolic time for $x$ and $k$ is a
hyperbolic time for $f^n(x)$, the intersection $V_n(x) \cap
f^{-k}(V_k(f^k(x)))$ coincides with the hyperbolic pre-ball
$V_{n+k}(x)$.
\end{remark}

The next lemma, which we borrow from \cite{OV07}, provides an
abstract criterium for non-lacunarity \emph{at almost every point}
of certain sequences of functions.

\begin{lemma}\cite[Proposition~3.8]{OV07}\label{non-lacunary criterium}
Let $T:M \to \N$ and $T_i:M \to \N$ , $i\in \N$ be measurable
functions and $\eta$ be a probability measure such that
$$
T(f^{T_i(x)}(x))\geq T_{i+1}(x)-T_{i}(x)
$$
at $\eta$-almost every $x \in M$. Assume $\eta$ is invariant under
$f$ and $T$ is integrable for $\eta$. Then $(T_i(x))_i$ is
non-lacunary for $\eta$-almost every $x$.
\end{lemma}

The application we have in mind is when $T_i=n_i$ is the sequence of
hyperbolic times, with $T=n_1$. In this case the assumption of the
lemma follows from the concatenation property in
Remark~\ref{r.concatenation}. Thus, we obtain

\begin{corollary}\label{c.hyperbolicislacunary}
If $\eta$ is an invariant expanding measure and $n_1(\cdot)$ is
$\eta$-integrable then the sequence $n_j(\cdot)$ is non-lacunary at
$\eta$-almost every point.
\end{corollary}

\subsection{Relative pressure}\label{Relative Pressure}

We recall the notion of topological pressure on non necessarily
compact invariant sets, and quote some useful properties. In fact,
we present two alternative characterizations of the relative
pressure, both from a dimensional point of view. See Chapter 4
$\S$11 and Appendix II of \cite{Pe97} for proofs and more details.

Let $M$ be a compact metric space, $f:M \to M$ be a continuous
transformation, $\phi:M \to \mathbb R$ be a continuous function, and
$\Lambda$ be an $f$-invariant set.

\vspace{.2cm} \noindent\emph{Relative pressure using partitions:}
Given any finite open covering $\cU$ of $\Lambda$, denote by $\cI_n$
the space of all $n$-strings $\underbar i=\{(U_0, \dots, U_{n-1}) :
U_i \in \cU\}$ and put $n(\underbar i)=n$. For a given string $\un
i$ set
$$
\underbar U =\underbar{U}(\underbar i)
        =\{x \in M : f^j(x) \in U_{i_j}, \;\text{for}\; j=0 \dots n(\underbar i)\}
$$
to be the cylinder associated to $\un i$ and $n(\un U)=n$ to be its
depth. Furthermore, for every integer $N \geq 1$, let $\cS_ N \cU$
be the space of all cylinders of depth at least $N$. Given $\alpha
\in \R$ define
\begin{equation*}\label{eq. alpha measure}
m_\al(f,\phi,\Lambda,\cU,N)= \inf_{\mathcal{G}} \Big\{
\sum_{{\mathrm{\underline U}} \in \mathcal{G}}
 e^{-\alpha n({\mathrm{\underline U}})+ S_{n({\mathrm{\underline U}})} \phi({\mathrm{\underline U}})} \Big\},
\end{equation*}
where the infimum is taken over all families $\mathcal{G}\subset
\cS_N \cU$ that cover $\Lambda$ and we write
$S_n\phi({\mathrm{\underline U}})=\sup_{x \in {\mathrm{\underline
U}}} S_n \phi(x)$. Let
$$
m_\al(f,\phi,\Lambda,\cU)
    = \lim_{N \to \infty} m_\al(f,\phi,\Lambda,\cU,N)
$$
(the sequence is monotone increasing) and
$$
P_\Lambda(f,\phi,\cU)
    = \inf{\{\alpha : m_\al(f,\phi,\Lambda,\cU)=0\}}.
$$

\begin{definition}\label{def. relative pressure1}
The \textit{pressure of $(f,\phi)$ relative to $\Lambda$} is
\begin{equation*}
P_\Lambda(f,\phi)= \lim_{\diam(\cU) \to 0} P_\Lambda(f,\phi,\cU).
\end{equation*}
\end{definition}
Theorem 11.1 in \cite{Pe97} states that the limit does exist, that
is, given any sequence of coverings $\cU_k$ of $L$ with diameter
going to zero, $P_L(f,\phi,\cU_k)$ converges and the limit does not
depend on the choice of the sequence.

\vspace{.2cm} \noindent\emph{Relative pressure using dynamical
balls:}

Fix $\vep>0$. Set $\cI_n= M\times\{n\}$ and $\cI =M \times \N$. For
every $\alpha \in \R$ and $N\geq 1$, define
\begin{equation}\label{eq. alpha measure dynamical balls}
m_\al(f,\phi,\Lambda,\vep,N) = \inf_{\mathcal{G}} \Big\{
\sum_{{(x,n)} \in \mathcal{G}}
  e^{-\alpha n+ S_{n} \phi({\mathrm B(x,n,\vep)})} \Big\},
\end{equation}
where the infimum is taken over all finite or countable families
$\mathcal{G}\subset \cup_{n \geq N}\cI_n$ such that the collection
of sets $\{B(x,n,\vep): (x,n)\in \mathcal G\}$ cover $\Lambda$. Then
let
$$
m_\al(f,\phi,\Lambda, \vep)
  = \lim_{N \to \infty} m_\al(f,\phi,\Lambda,\cU,N)
$$
(once more, the sequence is monotone increasing) and
$$
P_\Lambda(f,\phi,\vep) = \inf{\{\alpha:
m_\al(f,\phi,\Lambda,\vep)=0\}}.
$$
According to Remark 1 in \cite[Page 74]{Pe97} there is a limit when
$\vep\to 0$ and it coincides with the relative pressure:
\begin{equation*}\label{def. relative pressure2}
P_\Lambda(f,\phi)= \lim_{\vep \to 0} P_\Lambda(f,\phi,\vep).
\end{equation*}

\begin{remark}\label{r.centers}
Since $\phi$ is uniformly continuous, the definition of the relative
pressure is not affected if one replaces, in \eqref{eq. alpha
measure dynamical balls}, the supremum $S_n\phi(B(x,n,\vep))$ by the
value $S_n\phi(x)$ at the center point.
\end{remark}

The following properties on relative pressure, will be very useful
later. See Theorem~11.2 and Theorem~A2.1 in \cite{Pe97}, and also
\cite[Theorem~9.10]{Wa82}.

\begin{proposition}\label{non-compact variational principle}
Let $M$ be a compact metric space, $f:M \to M$ be a continuous
transformation, $\phi:M \to \mathbb R$ be a continuous function, and
$\Lambda$ be an $f$-invariant set. Then
\begin{enumerate}
\item $P_{\Lambda}(f,\phi)\geq \sup\left\{ h_\mu(f)+\int
\phi d\mu \right\}$ where the supremum is over all invariant
measures $\mu$ such that $\mu(\La)=1$. If $\Lambda$ is compact, the
equality holds.
\item $\Ptop (f,\phi)=\sup\{P_\Lambda(f,\phi),P_{M\setminus\Lambda}(f,\phi)\}$.
\end{enumerate}
\end{proposition}

The next proposition is probably well-known. We include a proof
since we could not find one in the literature.

\begin{proposition}\label{p.pressure.iteration}
Let $M$ be a compact metric space, $f:M \to M$ be a continuous
transformation, $\phi:M \to \mathbb R$ be a continuous function, and
$\Lambda$ be an $f$-invariant set. Then $P_\La(f^\ell,S_\ell
\phi)=\ell P_\La(f,\phi)$ for every $\ell \geq 1$.
\end{proposition}

\begin{proof}
Fix $\ell\geq 1$. By uniform continuity of $f$, given any $\rho>0$
there exists $\vep>0$ such that $d(x,y) < \vep$ implies
$d(f^j(x),f^j(y))<\rho$ for all $0 \le j < \ell$. It follows that
\begin{equation}\label{eq.balls}
B_{f}(x,\ell n,\vep)
 \subset B_{f^\ell}(x, n,\vep)
 \subset B_{f}(x,\ell n, \rho),
\end{equation}
where $B_{g}(x,n,\vep)$ denotes the dynamical ball for a map $g$.
This is the crucial observation for the proof.

First, we prove the $\ge$ inequality. Given $N\ge 1$ and any family
$\cG_\ell \subset \cup_{n \geq N} \cI_n$ such that the balls
$B_{f^\ell}(x,j,\vep)$ with $(x,j)\in\cG_\ell$ cover $\Lambda$,
denote
$$
\cG=\{(x,j\ell) : (x,j)\in\cG_\ell\}.
$$
The second inclusion in \eqref{eq.balls} ensures that the balls
$B_f(x,k,\rho)$ with $(x,k)\in\cG$ cover $\Lambda$. Clearly,
$$
\sum_{(x,j)\in\cG_\ell}
 e^{-\alpha \ell j + \sum_{i=0}^{j-1} S_\ell\phi(f^{i\ell}(x))}
= \sum_{(x,k)\in\cG} e^{-\alpha k + \sum_{i=0}^{k-1} \phi(f^i(x))}.
$$
Since $\cG_\ell$ is arbitrary, and recalling Remark~\ref{r.centers},
this proves that
$$
m_{\alpha\ell}(f^\ell,S_\ell\phi,\Lambda,\vep,N)
 \ge m_\alpha(f,\phi,\Lambda,\rho,N\ell).
$$
Therefore, $m_{\alpha\ell}(f^\ell,S_\ell\phi,\Lambda,\vep)
 \ge m_\alpha(f,\phi,\Lambda,\rho)$.
Then $P_\Lambda(f^\ell,S_\ell\phi,\vep) \ge \ell P_\Lambda(f,
\phi,\rho)$. Since $\vep\to 0$ when $\rho\to 0$, it follows that
$P_\Lambda(f^\ell,S_\ell\phi) \ge \ell P_\Lambda(f, \phi)$.

For the $\le$ inequality, we observe that the definition of the
relative pressure is not affected if one restricts the infimum in
\eqref{eq. alpha measure dynamical balls} to families $\cG$ of pairs
$(x,k)$ such that $k$ is always a multiple of $\ell$. More
precisely, let $m_\alpha^\ell(f,\phi,\Lambda,\vep, N)$ be the
infimum over this subclass of families, and let
$m_\alpha^\ell(f,\phi,\Lambda,\vep)$ be its limit as $N\to\infty$.

\begin{lemma}\label{l.multiplosdeell}
We have $m_\alpha^\ell(f,\phi,\Lambda,\vep)
 \le m_{\alpha-\rho}(f,\phi,\Lambda,\vep)$ for every $\rho>0$.
\end{lemma}

\begin{proof}
We only have to show that, given any $\rho>0$,
\begin{equation}\label{eq.menorquerho}
m_\alpha^\ell(f,\phi,\Lambda,\vep,N) \le
m_{\alpha-\rho}(f,\phi,\Lambda,\vep,N)
\end{equation}
for every large $N$. Let $\rho$ be fixed and $N$ be large enough so
that $N\rho > \ell(\alpha + \sup|\phi|)$. Given any
$\cG\subset\cup_{n \geq N} \cI_n$ such that the balls
$B_f(x,k,\vep)$ with $(x,k)\in\cG$ cover $\Lambda$, define $\cG'$ to
be the family of all $(x,k')$, $k'=\ell[k/\ell]$ such that
$(x,k)\in\cG$. Notice that
$$
-\alpha k' + S_{k'}\phi(x) \le -\alpha k + \alpha \ell + S_k\phi(x)
+ \ell \sup|\phi| \le (-\alpha+\rho) k + S_k\phi(x)
$$
given that $k\ge N$. The claim follows immediately.
\end{proof}

Let $\cG'$ be any family of pairs $(x,k)$ with $k\ge N\ell$ and such
that every $k$ is a multiple of $\ell$. Define $\cG_\ell$ to be the
family of pairs $(x,j)$ such that $(x,j\ell)\in\cG'$. The first
inclusion in \eqref{eq.balls} ensures that if the balls
$B_f(x,k,\vep)$ with $(x,k)\in\cG'$ cover $\Lambda$ then so do the
balls $B_{f^\ell}(x,j,\vep)$ with $(x,j)\in\cG_\ell$. Clearly,
$$
\sum_{(x,k)\in\cG'} e^{-\alpha k + \sum_{i=0}^{k-1} \phi(f^i(x))} =
\sum_{(x,j)\in\cG_\ell}
 e^{-\alpha \ell j + \sum_{i=0}^{j-1} S_\ell\phi(f^{i\ell}(x))}.
$$
Since $\cG_\ell$ is arbitrary, and recalling Remark~\ref{r.centers},
this proves that
$$
m_\alpha^\ell(f,\phi,\Lambda,\vep,N\ell)
 \ge m_{\alpha\ell}(f^\ell,S_\ell\phi,\Lambda,\vep,N).
$$
Taking the limit when $N\to\infty$ and using
Lemma~\ref{l.multiplosdeell},
$$
m_{\alpha-\rho}(f,\phi,\Lambda,\vep) \ge
m_\alpha^\ell(f,\phi,\Lambda,\vep)
 \ge m_{\alpha\ell}(f^\ell,S_\ell\phi,\Lambda,\vep).
$$
It follows that $\ell\big(P_\Lambda(f,\phi,\vep) + \rho\big) \ge
P_\Lambda(f^\ell, S_\ell\phi,\vep)$. Since $\rho$ is arbitrary, we
conclude that $\ell P_\Lambda(f,\phi,\vep) \ge P_\Lambda(f^\ell,
S_\ell\phi,\vep)$ and so $P_\Lambda(f^\ell,S_\ell\phi) \ge \ell
P_\Lambda(f, \phi)$.
\end{proof}

The next lemma will be used later to reduce some estimates for the
relative pressure to the case when $\phi\equiv 0$. Denote
$h_\Lambda(f) = P_\Lambda(f,0)$ for any invariant set $\Lambda$.

\begin{lemma}\label{l.reduction}
$P_\Lambda(f,\phi) \leq h_\Lambda(f) + \sup \phi$.
\end{lemma}

\begin{proof}
Let $\cU$ be any open covering of $M$ and $N \geq 1$. By definition,
$$
m_\al(f,\phi,\Lambda,\cU,N)
     = \inf_{\mathcal{G}} \Big\{\sum_{{\mathrm{\underline U}} \in \,\mathcal{G}}
      e^{-\alpha n({\mathrm{\underline U}})+ S_{n({\mathrm{\underline U}})}
      \phi({\mathrm{\underline U}})} \Big\},
$$
where the infimum is taken over all families $\mathcal{G}\subset
\cS_N \cU$ that cover $\Lambda$. Therefore,
$$
m_\al(f,\phi,\Lambda,\cU,N)
     \le \inf_{\mathcal{G}} \Big\{\sum_{{\mathrm{\underline U}} \in \,\mathcal{G}}
      e^{(-\alpha+\sup\phi) n({\mathrm{\underline U}})} \Big\}
      = m_{\al-\sup\phi}\,(f,0,\Lambda,\cU,N).
$$
Since $N$ and $\cU$ are arbitrary, this gives that
$P_\Lambda(f,\phi)\leq h_\Lambda(f)+\sup\phi$, as we wanted to
prove.
\end{proof}

\subsection{Natural extension and local unstable leaves}\label{Natural extension}

Here we present the natural extension associated to a non-invertible
transformation and recall some results on the existence of local
unstable leaves in the context of non-uniform hyperbolicity.

Let $(M,\cB,\eta)$ be a probability space and let $f$ denote a
measurable non-invertible transformation. Consider the space
$$
\hat M
    = \Big\{ (\dots, x_2, x_1, x_0) \in M^\mathbb N : f(x_{i+1})=x_i , \;\forall i \geq
    0\Big\},
$$
endowed with the metric $\hat d(\un x,\un y)= \sum_{i \geq 0} 2^{-i}
d(x_i,y_i)$, \;$\un x, \un y \in \hat M$ and with the sigma-algebra
$\hat \cB$ that we now describe. Let $\pi_i: \hat M \to M$ denote
the projection in the $i$th coordinate. Note also that $f^{-i}(\cB)
\subset \cB$ for every $i \geq 0$, because $f^i$ is a measurable
transformation. Let $\hat\cB_0$ be the smallest sigma-algebra that
contain the elements $\pi_i^{-1}(f^{-i}(\cB))$. The measure
$\hat\eta$ defined on the algebra $\bigcup_{i=0}^{\infty}
\pi_i^{-1}(f^{-i}\cB)$ by
$$
\hat\eta(E_i)= \eta(\pi_i(E_i)) \quad\text{for every}\; E_i \in
\pi_i^{-1}(f^{-i}(\cB)),
$$
admits an extension to the sigma-algebra $\hat\cB_0$. Let $\hat\cB$
denote the completion of $\hat\cB_0$ with respect to $\hat\eta$.
The \emph{natural extension} of $f$ is the transformation
$$
\hat f: \hat M \to \hat M,
    \quad
\hat f (\dots, x_2,x_1, x_0) = (\dots, x_2, x_1, x_0, f(x_0)),
$$
on the probability space $(\hat M, \hat \cB, \hat\eta)$. The measure
$\hat\eta$ is the unique $\hat f$-invariant probability measure such
that $\pi_*\hat\eta=\eta$. Furthermore, $\hat\eta$ is ergodic if and
only if $\eta$ is ergodic, and its entropy $h_{\hat\eta}(\hat f)$
coincides with $h_{\eta}(f)$. We refer the reader to \cite{Ro64} for
more details and proofs. For simplicity reasons, when no confusion
is possible we denote by $\pi$ the projection in the zeroth
coordinate and by $x_0$ the point $\pi(\hat x)$.

Given a local homeomorphism $f$ as above, the natural extension
$\hat f^{-1}$ is Lipschitz continuous: every $\hat x$ admits a
neighborhood $U_{\hat x}$ such that
$$
d( \hat f^{-1}(\hat y), \hat f^{-1}(\hat z))
    \leq \hat L(\hat x) \; d(\hat y, \hat z),
    \quad \forall \hat y, \hat z \in \hat f(U_{\hat x}),
$$
where $\hat L=L \circ \pi$. In the presence of asymptotic expanding
behavior it is possible to prove the existence of local unstable
manifolds passing through almost every point and varying measurably.
In fact, since $L$ is continuous bounded away from zero and
infinity, given an $\hat f$-invariant probability measure
$\hat\eta$, Birkhoff's ergodic theorem asserts that the limits
$$
\lim_{n\to\infty}
    \frac1n \sum_{j=0}^{n-1} \log \hat L(\hat f^{\pm j}\hat x)
$$
exist and coincide $\hat\eta$-almost everywhere. Given $\la>0$,
denote by $\hat B_\la$ the set of points such that the previous
limit is well defined and smaller than $-\lambda$.

\begin{proposition}\label{p.Pesinpointwise}
Assume that $\eta$ is an $f$-invariant probability measure such that
$$
\limsup_{n\to\infty}
    \frac1n \sum_{j=0}^{n-1} \log L(f^{j}(x)) <-\la<0
$$
almost everywhere. Given $\vep>0$ small, there are measurable
functions $\de_\vep$ and $\ga$ from $\hat B_\la$ to $\R_+$ and, for
every $\hat x \in \hat B_\la$, there exists an embedded topological
disk $\Wloc^u(\hat x)$ that varies measurably with the point $\hat
x$ and
\begin{enumerate}
\item For every $y_0 \in \Wloc^u(\hat x)$ there is a unique $\hat y \in \hat M$
    such that $\pi(\hat y)=y_0$ and
    $$
    d(x_{-n},y_{-n}) \leq \ga(\hat x) \, e^{-(\la-\vep) n}
    \; \forall n \geq 0;
    $$
\item If a point $\hat z \in \hat M$ satisfies
    $d(x,z) \leq \de_\vep(\hat x)/ \ga(\hat f^{-1}(\hat x))$ and
    $$
    d(x_{-n},z_{-n}) \leq \de_\vep(\hat x) e^{-(\la-\vep) n}, \,
    \forall n \ge 0
    $$
    then $z_0$ belongs to $\Wloc^u(\hat x)$;
\item If $\hWloc^u(\hat x)$ is the set of points $\hat y \in \hat M$
    given by (2) above then it holds that
        $$
        d(y_{-n},z_{-n}) \leq \ga(\hat x) \, e^{-(\la-\vep) n} d(y,z)
        $$
    for every $\hat y,\hat z \in \hWloc^u(\hat x)$ and every $n \geq 0$.
\end{enumerate}
\end{proposition}

\begin{proof}
Let $\vep>0$ be small enough such that the restriction of $f$ to any
ball of radius $\vep$ is injective. Given $\hat x \in \hat B_\la$,
consider the local unstable set
$$
\Wloc^u(\hat x)
    =\Big\{
    y \in M : \exists \hat y \in \hat M, \pi(\hat y)=y,
    d(y_{-n}, x_{-n}) \leq \vep, \; \forall n \geq 0
    \Big\}.
$$
By construction $\Wloc^u(\hat x)$ is non-empty, since it contains
$x$. Moreover, define $\hWloc^u(\hat x)$ as the set of points $\hat
y$ considered in the definition of $ \Wloc^u(\hat x)$. It is clear
that $\hat f^{-1}(\hWloc^u(\hat x)) \subset \hat W^u_\vep(\hat
f^{-1}(\hat x))$. We claim that $\Wloc^u(\hat x)$ contains an open
neighborhood of $x$ in $M$ and that there exists a constant
$\ga(\hat x)>0$ such that
$$
d(y_{-n},z_{-n}) \leq \ga(\hat x) \, e^{-(\la-\vep) n},
     \forall n \ge 1,
$$
for every $\hat y,\hat z \in \hWloc^u(\hat x)$. By hypothesis, there
exists $N=N_{\hat x}\ge 1$ such that
$$
\prod_{j=0}^{n-1} \hat L(\hat f^{-j}(\hat x))
        \leq e^{-\la n}, \; \forall n \geq N.
$$
Take $0<\de_\vep(\hat x)<\vep$ such that $f^{N}$ is invertible in a
neighborhood of ${x_{-N}}$ and that $B(x,\de_\vep(\hat x)) \subset
f^N(B(x_{-N},\vep))$. Moreover, by uniform continuity, there exists
$0<\vep_1<\vep$ such that $L(z) \leq L(z') \;e^{\vep}$ for every $z'
\in B(z,\vep_1)$. So, given $y,z \in B(x,\de_\vep(\hat x))$ there
are $\hat y,\hat z \in \hat M$ such that $d(y_{-n},z_{-n})\leq \vep$
for every $n \ge 0$, since
$$
d(y_{-n}, z_{-n})
    \leq e^{-(\la-\vep)n} \,d(y,z)
$$
for every $n \ge N$. This shows that $W_\vep^u(\hat x)$ contains the
ball $B(x,\de_\vep(\hat x))$ of radius $\de_\vep(\hat x)$ around $x$
in $M$ and that
$$
d(y_{-n}, z_{-n})
    \leq \ga(\hat x) \;e^{-(\la-\vep)n} \,d(y,z)
$$
for every $\hat y, \hat z \in W_\vep^u(\hat x)$ and $n \geq 0$,
where $\ga(\hat x)=L^{N_{\hat x}}$. Our choice on $\vep$ guarantees
that any $y \in W_\vep^u(\hat x)$ admits a unique $\hat y \in \hat
W_\vep^u(\hat x)$ such that $\pi(\hat y)=y$. This shows that the
projection $\pi_{\hat x} : \hWloc^u(\hat x) \to W_\vep^u(\hat x)$ is
an homeomorphism between topological disks and completes the proof
of items (1) and (3) in the proposition. On the other hand, if $\hat
z$ satisfies the requirements in (2) then clearly $d(x_{-n},z_{-n})
\leq \vep$ for all $n\geq 00$, which imply that $z \in \Wloc^u(\hat
x)$.Since the measurability of $\gamma$ and $\de_\vep$ follows from
the one of $N_{\hat x}$, the proof of the proposition in now
complete.
\end{proof}

We shall omit the dependence of $\Wloc^u(\hat x)$ on $\la$ and
$\vep$ for notational simplicity. Since local unstable leaves vary
measurably with the point, there are compact sets of arbitrary large
measure, referred as \emph{hyperbolic blocks}, restricted to which
the local unstable leaves passing through those points vary
continuously. More precisely,

\begin{corollary}\label{c.Pesinblocks}
There are countably many compact sets $(\hat \La_i)_{i \in \N}$
whose union is a $\hat\eta$-full measure set and such that the
following holds: for every $i \geq 1$ there are positive numbers
$\vep_i\ll 1$, $\la_i$, $r_i$, , $\ga_i$ and $R_i$ such that for
every $\hat x \in \hat\La_i$ there exists an embedded submanifold
$\Wloc^u(\hat x)$ in $M$ of dimension $m$, and
\begin{enumerate}
\item If $y_0 \in \Wloc^u(\hat x)$ then there is a unique $\hat y \in \hat
        M$ such that for every $n \geq 1$
        \begin{equation*}
        d(x_{-n},y_{-n}) \leq r_i e^{-\vep_i n} \quad\text{and}\quad
        d(x_{-n},y_{-n}) \leq \ga_i e^{-\la_i n};
        \end{equation*}
\item For every $0<r \leq r_i$ the set $\Wloc^u(\hat y) \cap B(x_0,r)$ is
        connected and the map
        $$
        B(\hat x, \vep_i r) \cap \hat\La_i\ni \hat y \mapsto \Wloc^u(\hat y) \cap B(x_0,r)
        $$
        is continuous (in the Hausdorff topology);
\item If $\hat y$ and $\hat z$ belong to $B(\hat x, \vep_i r) \cap \hat\La_i$
        then either $\Wloc^u(\hat y) \cap B(x_0,r)$ and $\Wloc^u(\hat z) \cap B(x_0,r)$
        coincide or are disjoint; in the later case, if  $\hat y \in {\hat W}^u(\hat z)$ then
        $d(y_0,z_0)>2r_i$;
\item If $\hat y \in \hat\La_i \cap B(\hat x, \vep_i r)$ then $\Wloc^u(\hat y)$
        contains the ball of radius $R_i$ around $\Wloc^u(\hat y) \cap
        B(x_0,r)$.
\end{enumerate}
\end{corollary}

\section{Conformal measures}\label{Conformal Measure}

The \emph{Ruelle-Perron-Fr\"obenius transfer operator} $\cL_\phi:
C(M) \to C(M)$ associated to $f:M\to M$ and $\phi:M\to\real$ is the
linear operator defined on the space $C(M)$ of continuous functions
$g:M\to\real$ by
$$
\cL_\phi g(x) = \sum_{f(y)=x} e^{\phi(y)}g(y).
$$
Notice that $\cL_\phi g$ is indeed continuous if $g$ is continuous,
because $f$ is a local homeomorphism. It is also easy to see that
$\cL_\phi$ is a bounded operator, relative to the norm of uniform
convergence in $C(M)$:
$$
\|\cL_\phi\| \le  \max_{x \in M}\# f^{-1}(x) \; e^{\sup|\phi|}.
$$
The dual operator $\cL^*_\phi$ acts on the Borel measures of $M$ by
Consider the dual operator $\cL^*_\phi:\cM(M)\to\cM(M)$ acting on
the space $\cM(M)$ of Borel measures in $M$ by
$$
\int g \, d(\cL_\phi^*\eta) = \int (\cL_\phi g) \, d\eta
$$
for every $g\in C(M)$. Let $\la_0=r(\mathcal L_\phi)$ be the
spectral radius of $\cL_\phi$. In this section we prove the
following result:

\begin{theorem}\label{thm. conformal measure}
There exists $k \geq 1$, $r(\cL_\phi)=\la_0 \geq \la_1 \geq \dots
\geq \la_k \geq e^{h(f)+\inf \phi}$ real numbers and expanding
conformal probability measures $\nu_0, \nu_1, \dots, \nu_k$ such
that
$$
\cL_\phi^* \nu_i=\la_i \nu_i, \; \forall \,0 \leq i \leq k,
    \quad\text{and}\quad
\bigcup_{i=0}^k \supp(\nu_i)= \overline H.
$$
Moreover, each $\nu_i$ is a non-lacunary Gibbs measure and has a
Jacobian with respect to $f$ given by $J_{\nu_i} f =\la_i
e^{-\phi}$. If $f$ is topologically mixing then $\nu_0$ is an
expanding conformal measure such that $\supp \nu_0=\overline H=M$.
\end{theorem}

\subsection{Eigenmeasures of the transfer operator}

The following lemma asserts that any positive eigenmeasure for the
dual of the Ruelle-Perron-Frobenius operator is a conformal measure.
Its proof is quite standard: see, for instance,
\cite[Lemma~4.1]{OV07}.

\begin{lemma}\label{l.Jacobian}
Suppose $\nu$ is a Borel probability such that
$\cL_\phi^*\nu=\lambda\nu$ for some $\lambda>0$. Then the Jacobian
of $\nu$ with respect to $f$ exists and is given by $J_\nu f =
\lambda e^{-\phi}$.
\end{lemma}

The proof of the next lemma is analogous to \cite[Lemma~4.2]{OV07}.

\begin{lemma}\label{l.eigenvalue}
The spectral radius $\la_0$ of the operator $\cL_\phi$ is at least
$e^{h(f)+\inf\phi}$ and it is an eigenvalue for the dual operator
$\cL_\phi^*$.
\end{lemma}

Throughout, let $\la$ denote a fixed eigenvalue of $\cL_\phi^*$
larger than $e^{h(f)+\inf \phi}$, let $\nu$ be any eigenmeasure of
$\cL_\phi^*$ associated to $\la$ and set $P=\log\la$. The only
property of $\la$ that we shall use is that $\la>e^{\log q+\sup\phi
+\vep_0}$. From Lemma~\ref{l.Jacobian} we get that
\begin{equation}\label{eq.Jacobian}
J_\nu f (x)
 = \lambda_0 e^{-\phi(x)}
 > e^{\log q +\vep_0} > q
 \quad\text{for all $x\in M$.}
\end{equation}
This property will allow us to prove that $\nu$-almost every point
spends at most a fraction $\ga$ of time inside the domain $\cA$
where $f$ may fail to be expanding. As we will see later, in
Lemma~\ref{Gibbs imply equilibrium}, $\log \la=\Ptop(f,\phi)$. This
determines completely the spectral radius of $\cL_\phi$ as the
\emph{unique} eigenvalue of $\cL_\phi^*$ larger than the lower bound
above. Consequently all the eigenvalues $\la_i$ given by
Theorem~\ref{thm. conformal measure} are equal and coincide with
$\la_0=r(\cL_\phi)$ and $\frac1k \sum_{j=0}^k \, \nu_i $ is an
expanding conformal measure whose support coincides with the closure
of the set $H$. The later is the conformal measure referred at
Theorem~\ref{Thm. Equilibrium States2}.

\subsection{Expanding structure}\label{Expanding Structure}

Here we prove that any eigenmeasure $\nu$ as above is expanding and
has integrable first hyperbolic time. Given $n\ge 1$, let $B(n)$
denote the set of points $x \in M$ whose frequency of visits to
$\cA$ up to time $n$ is at least $\ga$, that is,
$$
B(n) =\Big\{x \in M : \frac{1}{n} \# \{ 0\leq j \leq n-1 :f^j(x)\in
\cA\} \geq \ga \Big\}.
$$

\begin{proposition}\label{p.recurrence}
The measure $\nu(B(n))$ decreases exponentially fast as $n$ goes to
infinity. Consequently, $\nu$-almost every point belongs to $B(n)$
for at most finitely many values of $n$.
\end{proposition}

\begin{proof}
The strategy is to cover $B(n)$ by elements of the covering
$\cP^{(n)}=\bigvee_{j=0}^{n-1}f^{-j}\cP$ which, for convenience,
will be referred to as cylinders. Then, the estimate relies on an
upper bound for the measure of each cylinder, together with an upper
bound on the number of cylinders corresponding to large frequency of
visits to $\cA$.

Since $f^n$ is injective on every $P \in \cP^{(n)}$ then we may use
\eqref{eq.Jacobian} to conclude that
$$
1 \ge \nu(f^n(P))
 = \int_{P} J_\nu f^n \,d\nu
 = \int_{P} \prod_{j=0}^{n-1} (J_\nu f \circ f^j)  d\nu
 \ge e^{(\log q +\vep_0)n} \nu(P).
$$
This proves that $\nu(P) \le e^{-(\log q + \vep_0)n}$ for every
$P\in\cP^n$. Since $B(n)$ is contained in the union of cylinders
$P\in\cP^n$ associated to itineraries in $I(\gamma,n)$, we deduce
from our choice of $\ga$ after Lemma ~\ref{l.combinatorio} that
$$
\nu(B(n)) 
          \leq \# \, I(\ga,n) e^{-(\log q+\vep_0)n}
          \leq e^{-\vep_0 n/2},
$$
for every large $n$. This proves the first statement in the lemma.
The second one is a direct consequence, using the Borel-Cantelli
lemma.
\end{proof}

\begin{corollary}\label{c.nu.expanding}
The measure $\nu$ is expanding and satisfies $\int n_1
\,d\nu<\infty.$
\end{corollary}

\begin{proof}
By Proposition~\ref{p.recurrence}, almost every point $x$ is outside
$B(n)$ for all but finitely many values of $n$. Then, in view of our
choice \eqref{eq. relation expansion},
$$
\sum_{j=0}^{n-1} \log L(f^j(x))
    \leq \ga \log L + (1-\ga) \log \si^{-1}
    \leq -2c
$$
if $n$ is large enough. In view of Lemma~\ref{positive density},
this proves that $\nu$-almost every point has infinitely many
hyperbolic times (positive density at infinity). In other words,
$\nu$ is expanding. Moreover, using Proposition~\ref{p.recurrence}
once more,
$$
\int n_1 d\nu
 = \sum_{n=0}^\infty \nu(\{x: n_1(x) > n\})
 \le 1 + \sum_{n=1}^\infty \nu(B(n))
 < \infty,
$$
as we claimed.
\end{proof}

\subsection{Gibbs property}\label{Gibbs property}

Now we prove that $\nu$ satisfies a Gibbs property at hyperbolic
times. Later we shall see that hyperbolic times form a non-lacunary
sequence, almost everywhere, and then it will follow that $\nu$ is a
non-lacunary Gibbs measure.

\begin{lemma}\label{l.uniform ball measure}
The support of $\nu$ is an $f$-invariant set contained in the
closure of $H$. For any $\rho>0$ there exists $\xi>0$ such that
$\nu(B(x,\rho))\geq \xi$ for every $x \in \supp(\nu)$.
\end{lemma}

\begin{proof}
Since $\nu$ is expanding, it is clear $\supp(\nu)\subset \overline
H$. Let $x \in M$. Since $f$ is a local homeomorphism, the relation
$V=f(W)$ is a one-to-one correspondence between small neighborhoods
$W$ of $x$ and small neighborhood $V$ of $f(x)$. Moreover,
$$
\nu(V) = \int_W J_\nu f \, d\nu.
$$
is positive if and only if $\nu(W)>0$, because the Jacobian is
bounded away from zero and infinity. This proves that the support is
invariant by $f$. The second claim in the lemma is standard. Assume,
by contradiction, that there exists $\rho>0$ and a sequence
$(x_n)_{n\geq 1}$ in $\supp(\nu)$ such that $\nu(B(x_n,\rho))\to 0$
as $n \to\infty$. Since $\supp(\nu)$ is compact set, the sequence
must accumulate at some point $z \in \supp(\nu)$. Then
\[
\nu(B(z,\rho))\leq \liminf_{n \to \infty} \nu(B(x_n,\rho))=0,
\]
which contradicts $z \in \supp(\nu)$. This completes the proof of
the lemma.
\end{proof}

\begin{lemma}\label{l.Gibbs.hyp.times} There exists $K>0$ such that,
if $n$ is a hyperbolic time for $x\in \supp(\nu)$ then
$$
K^{-1}
 \leq \frac{\nu(B(x,n,\delta))}{e^{-P n + S_{n}\phi(y)}}
 \leq K,
$$
for every $y \in B(x,n,\delta)$.
\end{lemma}

\begin{proof}
Since $f^n \mid B(x,n,\delta)$ is injective, we get from the
previous lemma that
$$
\xi(\de) \leq \nu(B(f^n(x),\de))=\int_{B(x,n,\delta)} J_\nu f^n
\,d\nu \leq 1
$$
for every $x\in\supp(\nu)$. Then, the bounded distortion property in
Corollary~\ref{lem. bounded distortion} applied to the H\"older
continuous function $J_\nu f=\la e^{-\phi}$ gives that
$$
K_0^{-1} \xi(\delta)
 \le \nu(B(x,n,\delta)) \la^n e^{-S_n\phi(y)}
 \le K_0
$$
for every $y \in B(x,n,\delta)$. Recalling that $P=\log \la$, this
gives the claim with $K=K_0 \;\xi(\de)^{-1}$.
\end{proof}

\begin{remark}\label{rmk. smaller balls}
The same proof gives a somewhat stronger result: for $\nu$-almost
every $x$ and any $0<\vep \leq \de$, there exists $K(\vep)>0$ such
that
$$
K^{-1}(\vep) \leq \frac{\nu(B(x,n,\vep))}
                       {e^{-P n + S_n\phi(x)}}
             \leq K(\vep).
$$
if $n$ is a hyperbolic time for $x$. It suffices to take
$K(\vep)=K_0\xi(\vep)^{-1}$.
\end{remark}

We proceed with the proof of Theorem~\ref{thm. conformal measure}.
We have proven that any eigenmeasure $\nu$ for $\cL_\phi$ associated
to an eigenvalue $\la\geq e^{h(f)+\inf\phi}$ is necessarily
expanding, satisfies the Gibbs property at hyperbolic times and has
a Jacobian $J_\nu f=\la e^{-\phi}$. Furthermore,
Lemma~\ref{l.eigenvalue} guarantees that the spectral radius $\la_0$
is an eigenvalue of the operator $\cL_\phi$. Let $\nu_0$ denote any
such eigenmeasure. If $f$ is topologically mixing then $\supp
\nu_0=\overline H=M$. Indeed, given an open set $U$ there exists
$N\geq 1$ such that $f^N(U)=M$. Since $J_{\nu_0} f$ is bounded from
zero and infinity then clearly $\nu_0(U)>0$, which proves our claim.
Hence, to prove Theorem~\ref{thm. conformal measure} we are left to
show that there are finitely many eigenmeasures of $\cL_\phi^*$
associated to eigenvalues greater or equal to $e^{h(f)+\inf\phi}$
whose union of their supports coincide with $\overline H$. Given an
$f$-invariant compact set $\La$ we denote by $\cL_\La: C(\La) \to
C(\La)$ the restriction of the operator $\cL_\phi$ to the space of
continuous functions $C(\La)$.

\begin{lemma}\label{l.full.conformal}
There are finitely many $\la_0 \geq \la_1 \geq \dots \geq \la_k \geq
 e^{h(f)+\inf \phi}$ and probability measures $\nu_0, \nu_1,
\dots, \nu_k$ such that $\cL_\phi^* \nu_i=\la_i \nu_i$, for every $0
\leq i \leq k$, and that the union of their supports coincides with
the closure of the set $H$.
\end{lemma}

\begin{proof}
We obtain the desired finite sequence of conformal measures using
the ideas involved in the proof of Lemma~\ref{l.eigenvalue}
recursively. Indeed, Lemma~\ref{l.eigenvalue},
Corollary~\ref{c.nu.expanding} and Lemma~\ref{l.Gibbs.hyp.times}
assert that there exists an expanding conformal measure $\nu_0$ such
that $\cL_\phi^* \nu_0=\la_0 \nu_0$ and satisfies the Gibbs property
at hyperbolic times. Clearly $\supp(\nu_0)$ is an invariant set
contained in $\overline H$.

If $\supp(\nu_0) = \overline H$ then we are done. Otherwise we
proceed as follows. As we shall see in Lemma~\ref{support Gibbs},
the interior of the support of any expanding conformal measure $\nu$
is non-empty and contains almost every point in a ball of radius
$\de$ (depending only on $f$ and $c$). Consider the non-empty
compact invariant set $K_1= M \setminus interior(\supp(\nu_0))$ and
set $\la_1=r(\cL_{K_1}) \leq \la_0$. It is easy to check that $\la_1
\geq  e^{h(f)+\inf\phi}$. Then we may argue as in the proof of
Lemma~\ref{l.eigenvalue}: the cone of strictly positive functions in
$K_1$ is disjoint from the subspace $\{\cL_\phi g - \la g : g \in
C(K_1)\}$ and so there exists a probability measure $\nu_1$ such
that $\cL_\phi^* \nu_1=\la_1\nu_1$ whose support $\supp(\nu_1)$ is
contained in $K_1$. Since $\la_1 \geq  e^{h(f)+\inf\phi}$ then
$\nu_1$ is also expanding and its support must also contain a ball
of radius $\de$ in its interior.

Since $M$ is compact this procedure will finish after a finite
number of times. Hence there are finitely many compact sets $K_0,
\dots, K_k$ and expanding measures $\nu_0, \dots, \nu_k$ such that
$\supp(\nu_i) \subset K_i$ and $\overline H =\bigcup_i
\supp(\nu_i)$. This completes the proof of the lemma.
\end{proof}

For any conformal measure $\nu_i$ as above, we prove in
Proposition~\ref{prop. acim}) that there are finitely many invariant
ergodic measures that are absolutely continuous with respect to
$\nu_i$, that their densities are bounded from above and that their
basins cover $\nu_i$-almost every point. Hence, the non-lacunarity
of the sequence of hyperbolic times will be a consequence of
Lemma~\ref{non-lacunary criterium}. So, up to the proof of
Proposition~\ref{prop. acim}, this shows that each $\nu_i$ is a
non-lacunary Gibbs measure and completes the proof of
Theorem~\ref{thm. conformal measure}.

\section{Absolutely continuous invariant measures}\label{Absolutely Continuous Measures}

In this section we analyze carefully the Cesaro averages
\[
\nu_n=\frac{1}{n} \sum_{j=0}^{n-1} f_*^j\nu,
\]
and prove that every weak$^*$ accumulation point is absolutely
continuous with respect to $\nu$. It is well known, and easy to
check, that the accumulation points are invariant probabilities. In
the topologically mixing setting we also prove that there is a
unique absolutely continuous invariant measure and that it satisfies
the non-lacunar Gibbs property. The precise statement is

\begin{proposition}\label{prop. acim}
There are finitely many invariant, ergodic probability measures
$\mu_1, \mu_2, \dots ,\mu_k$ that are absolutely continuous with
respect to $\nu$ and such any absolutely continuous invariant
measure is a convex linear combination of $\mu_1, \mu_2, \dots
,\mu_k$. In addition, the measures $\mu_i$ are expanding and the
densities $d\mu_i/d\nu$ are bounded away from infinity. Moreover,
the union of the basins $B(\mu_i)$ cover $\nu$-almost every point in
$M$. If $f$ is topologically mixing then there is a unique
absolutely continuous invariant measure and it is a non-lacunary
Gibbs measure.
\end{proposition}

\subsection{Existence and finitude}\label{sec. upper bound and finitude}

First we prove that every accumulation point of $(\nu_n)_{n \geq 1}$
is absolutely continuous invariant measure with bounded density. For
every $n \in \N$ it holds that
$$
H_n^c \subset
    \Big\{n_1 (\cdot)> n\Big\}
    \bigcup \Big[ \bigcup_{k=0}^{n-1} H_k \cap f^{-k}(\{n_1 (\cdot)> n - k\}) \Big].
$$
In particular, we can use the inclusion above to write
$$
\nu_n\leq \mu_n+ \frac{1}{n}\sum_{j=0}^{n-1} \eta_j,
$$
where
$$
\mu_n
    =\frac{1}{n}\sum_{j=0}^{n-1} f_*^j(\nu \mid {H_j})
\quad\text{and}\quad \eta_j
    =\sum_{l=0}^\infty f^{l}_*\big(f^j_*(\nu \mid H_{j})| \{n_1>l\}\big).
$$

\begin{lemma}\label{density control}
There exists $C_2>0$ such that for every positive integer $n$ the
measures $f^n_*(\nu\mid H_n)$, $\mu_n$ and $\nu_n$ are absolutely
continuous with respect to $\nu$ with densities bounded from above
by $C_2$. Moreover, the same holds for every weak$^*$ accumulation
point $\mu$ of $(\nu_n)_{n \geq 1}$.
\end{lemma}

\begin{proof}
Let $A$ be any measurable set of small diameter, say
$\diam(A)<\de/2$, and such that $\nu(A)>0$. First we claim that
there is $C_2>0$ such that
$$
f^n_*(\nu\mid H_n)(A) \leq C_2 \, \nu(A), \quad \forall n \geq 1.
$$
Observe that either $f^n_*(\nu|H_n)(A)=0$, or $A$ is contained in a
ball $B=B(f^n(x),\de)$ of radius $\de$ for some $x\in H_n$. In the
first case we are done. In the later situation we compute
$$
f^n_*(\nu \mid H_n)(A)
    =\nu(f^{-n}(A) \cap H_n)
    = \sum_{i} \nu(f_{i}^{-n}(A \cap B)),
$$
where the sum is over all hyperbolic inverse branches $f_i^{-n}:B
\to V_i$ for $f^n$. Recall that the $\nu$-measure of any positive
measure ball of radius $\de$ is at least $\xi(\de)>0$ by
Lemma~\ref{l.uniform ball measure}. Thus, by bounded distortion
$$
f^n_*(\nu \mid H_n)(A)
    \leq K_0 \sum_{i} \frac{\nu(A)}{\nu (B)} \, \nu(V_i)
    \leq K_0 \,\xi(\de)^{-1}\, \nu(A),
$$
which proves our claim with $C_2=K_0 \,\xi(\de)^{-1}$. It follows
from the arbitrariness of $A$ that both $f^n_*(\nu \mid H_n)$ and
$\mu_n$ are absolutely continuous with respect to $\nu$ with density
bounded from above by $C_2$.

Similar estimates on the density of $\eta_n$ hold using that
$\{n_1>n\} \subset B(n)$, there are at most $e^{c_\ga n}$ cylinders
in $B(n)$, and that $J_\nu f^n > e^{(\log q + \vep_0) n}$ on each of
one of them. Indeed,
$$
((f^{l}_*\nu) | \{n_1>l\})(A)
    \leq \sum_{\substack{ P \in \cP^{(l)} \\ P \cap B(l) \neq \emptyset }}
        \nu (f^{-l}(A) \cap P)
    \leq \# B(l) \; e^{-(\log q + \vep_0) l} \nu(A)
$$
for every $l \geq 1$ and every measurable set $A$. Using that
$df^n_*(\nu \mid H_n)/d\nu \leq K_0 \,\xi(\de)^{-1}$ and summing up
the previous terms one concludes that
$$
\eta_j(A)
    \leq K_0 \,\xi(\de)^{-1} \sum_{l=0}^\infty e^{-\frac{\vep_0}{4}l} \;
    \nu(A), \, \forall j \geq 1.
$$
This shows that (up to replace $C_2$ by a larger constant) the
measures $\nu_n$ are also absolutely continuous with respect to
$\nu$ and that $d\nu_n/d\nu$ is bounded from above by $C_2$. The
second assertion in the lemma is an immediate consequence by
weak$^*$ convergence.
\end{proof}

The following lemma, whose proof explores the generating property of
hyperbolic pre-balls, plays a key role in proving finitude of
equilibrium states.

\begin{lemma}\label{support Gibbs} If $G$ is an
$f$-invariant set such that $\nu(G)>0$ then there is a disk $\Delta$
of radius $\delta /4$ so that $\nu(\Delta \backslash G)=0$.
\end{lemma}

\begin{proof}
In the case that $\nu$ coincides with the Lebesgue measure this
corresponds to \cite[Lemma 5.6]{ABV00}. Since the argument will be
used later on, we give a brief sketch of the proof.

Let $\vep>0$ be small. Take a compact $K$ and an open set $O$ such
that $K \subset G \cap H \subset O$ and $\nu(O\setminus K) < \vep
\nu(K)$. Set $n_0 \in \N$ such that $B(x,n,\de) \subset O$ for any
$x \in K \cap H_n$. If $n(x)$ denotes the first hyperbolic time of
$x$ larger than $n_0$ then
$$
K \subset \bigcup_{x \in K} B(x,n(x),\de/4) \subset O.
$$
Set $V(x)=B(x,n(x),\de)$ and $W(x)=B(x,n(x),\de/4)$. Since $K$ is
compact it is covered by finite open sets $(W(x))_{x \in X}$ for
some family $X=\{x_1, \dots, x_k\}$. Now we proceed recursively and
define
$$
n_1=\inf\{ n(x): x \in X \}
    \quad\text{and}\quad
X_1=\{x \in X : n(x)=n_1\}
$$
and, assuming that $n_i$ and $X_i$ are well defined for $1 \leq i
\leq m-1$, set
$$
n_m=\inf\{ n(x): x \in X \setminus (X_1 \cup \dots \cup X_{m-1}) \}
    \quad\text{and}\quad
X_m=\{x \in X : n(x)=n_m\}
$$
up to some finite positive integer $s$. Let $\ti X_1 \subset X_1$ be
a maximal family of points with pairwise disjoint $W(\cdot)$
elements. Moreover, given $\ti X_i \subset X_i$ for $1 \leq i \leq
m-1$ let $\ti X_m \subset X_m$ maximal such that every $W(x)$, $x
\in \ti X_m$, does not intersect any element $W(y)$ for some $y \in
\ti X_1 \cup \dots \ti X_m$. If $\ti X= \cup\{\ti X_i : 1 \leq i
\leq s\}$ then the dynamical balls $W(x)$, $x \in \ti X$, are
pairwise disjoint (by construction). It is also easy to see that for
every $y \in X$ there exists $x \in \ti X$ such that $W(y) \subset
V(x)$. Hence
$$
\nu\Big(\bigcup_{x \in \ti X} W(x) \setminus K\Big)
    \leq \nu(O \setminus K)
    < \vep \nu(K)
$$
and, by the bounded distortion property,
$$
\nu\Big(\bigcup_{x \in \ti X} W(x)\Big)
    \geq \tau \nu\Big(\bigcup_{x \in \ti X} V(x)\Big)
$$
for some $\tau>0$. We conclude immediately that there exists $x \in
\ti X$ such that
$$
\frac{\nu(W(x) \setminus G)}{\nu(W(x))} \leq \frac{\nu(W(x)
\setminus K)}{\nu(W(x))} < \tau^{-1}\vep.
$$
Using the bounded distortion of $f^n$ restricted to the dynamical
ball $W(x)$ once more it follows that
$$
\nu(B \setminus f^{n}(G))< \tau^{-1} K_0 \vep,
$$
where $B$ is a ball of radius $\de/4$ around $f^n(x)$. Since $\vep$
was arbitrary and $G$ is invariant then there exists a sequence
$\De_n$ of balls of radius $\de/4$ such that $\nu(\De_n \setminus G)
\to 0$ as $n\to\infty$. By compactness, the sequence $(\De_n)_n$
accumulate on a ball $\De$ that satisfies the requirements of the
lemma.
\end{proof}

We are now in a position to show that there are finitely many
distinct ergodic measures $\mu_1,\mu_2,\dots, \mu_k$ absolutely
continuous with respect to $\nu$. Indeed, let $\mu$ be any invariant
measure that is absolutely continuous. Then, either $\mu$ is ergodic
or there are disjoint invariant sets $I_1$ and $I_2$ of positive
$\nu$-measure such that
$\mu(\cdot)
    =a_1 \mu(\cdot \cap I_1 )/\mu(I_1) + a_2 \mu(\cdot \cap
    I_2)/\mu(I_2),$
where $a_i=\mu(I_i)$. In the later case it is also clear that each
of the measures involved in the sum is absolutely continuous with
respect to $\nu$. Repeating the process one obtains that $\mu$ can
be written as linear convex combination of ergodic absolutely
continuous invariant measures $\mu_1,\mu_2,\dots, \mu_k$. Indeed,
since $M$ is compact the previous lemma implies that this process
will stop after a finite number of steps (depending only on $\de$)
with each $\mu_i$ ergodic. It is also clear from the construction
that each $\mu_i$ is expanding and that their basins cover almost
every point.

\subsection{Invariant non-lacunary Gibbs measure}

Through the rest of this section assume that $f$ is topologically
mixing. Here we prove that there is a unique invariant measure $\mu$
absolutely continuous with respect to $\nu$ and that it is a
non-lacunary Gibbs measure. This will complete the proof of
Proposition~\ref{prop. acim}. We begin with a couple of auxiliary
lemmas. Let $\theta>0$ and $\de>0$ be given by Lemmas~\ref{positive
density} and~\ref{delta}.

\begin{lemma}
There exists a constant $\tau_0>0$, and for any $n$ there is a
finite subset $\hat H_n$ of $H_n$ such that the dynamical balls
$B(x,n,\de/4)$, $x \in \hat H_n$, are pairwise disjoint and their
union $W_n$ satisfies $ \nu (W_n) \geq \tau_0 \nu(H_n)$.
\end{lemma}

\begin{proof}
This lemma is a direct consequence of Lemma~3.4 in \cite{ABV00}.
Indeed, if $\om=f^n_*(\nu\mid \cup \{B(n,x,\de/4) : x \in H_n\})$,
$\Om=f^n(H_n)=M$ and $r=\de$ in that lemma then there exists a
finite set $I \subset f^n(H_n)$ such that the pairwise disjoint
union $\De_n$ of balls of radius $\de/4$ around points in $I$
satisfies
$$
\om \big(\De_n \cap f^n(H_n)\big)
    \geq \tau_0 \,\om (f^n(H_n)).
$$
Set $\hat H_n= H_n \cap f^{-n}(I)$. As the restriction of $f^n$ to
any dynamical ball $B(x,n,\de/4)$, $x \in \hat H_n$ is a bijection
it is easy to see that these dynamical balls are pairwise disjoint.
Furthermore, their union $W_n$ satisfies $\nu \big( W_n \big) \geq
\tau_0\, \nu (H_n)$. This completes the proof of the lemma.
\end{proof}

In the remaining of the section, let $\mu$ be an arbitrary
accumulation point of the sequence $(\nu_n)_n$ and $(n_k)_k$ be a
subsequence of the integers such that
$$
\mu=\lim_{k\to\infty} \nu_{n_k}.
$$
In the next lemmas we prove that the density $d\mu/d\nu$ is bounded
away from zero in some small disk and use this to deduce the
uniqueness of the equilibrium state and the non-lacunar Gibbs
property.

\begin{lemma}\label{l.density lower bound}
There exists $C_1>0$ and a small disk $D(x)$ around a point $x$ in
$M$ such that the density $d\mu/d\nu$ in the disk $D(x)$ is bounded
from below by $C_1$.
\end{lemma}

\begin{proof}
Given a small $\vep>0$ we construct a disk $D(x)$ of radius smaller
than $\vep$ where the assertion above holds. Let $W_j$ and $\hat
H_j$ be given by the previous lemma and let $W_{j,\vep} \subset W_j$
denote the preimages by $f^j$ of the disks $\De_{j,\vep}$ of radius
$\de/4-\vep$ around points in $f^j(\hat H_j)$. Lemma~\ref{lem.
bounded distortion} implies that
$$
\frac{\nu(W_{j,\vep})}{\nu(W_j)}
    \geq K_0^{-1} \frac{\nu(\De_{j,\vep})}{\nu(\De_j)},
$$
where the right hand side is larger than some uniform positive
constant $\tau_1$ that depends only on the radius of the disks
$\De_{j,\vep}$ (recall Lemma~\ref{l.uniform ball measure}). Observe
also that Corollary~\ref{c.definite fraction} with $A=M$ implies
that
$$
\frac{1}{n}\sum_{j=0}^{n-1} \nu (H_j) \geq \theta/2
$$
for every large $n$. This shows that there is a positive constant
$\tau_2$ such that the measures $\nu^\vep_n$ satisfy
$\nu^\vep_{n}(M) \geq \tau_2$ for every large $n$, where
$$
\nu^\vep_{n}=
    \frac{1}{n}\sum_{j=0}^{n-1} f_*^j (\nu|{W_{j,\vep}}).
$$
Thus, there exists a subsequence of $(\nu^\vep_{n_k})_k$ that
converge to some measure $\nu^\vep_\infty$ and
$$
\supp(\nu^\vep_\infty)
    \subset \bigcap_{n \geq 1} \big(\overline{ \bigcup_{j \geq n} \De_{j,\vep}}\big).
$$
Choose $x \in \supp(\nu^\vep_{\infty})$ and a disk $D(x)$ of radius
smaller than $\vep$ around $x$ such that $\nu^\vep_{\infty}(\partial
D(x))=0$. By construction, $D(x)$ is contained in every disk of
$\De_j$ such that the corresponding disk of $\De_{j,\vep}$
intersects $D(x)$. Let $\ti\De_j$ denote the pairwise disjoint union
of disks in $\De_j$ that contain $D(x)$ and $\ti W_j$ be defined
accordingly as the preimages of $\ti \De_j$. It is clear that $\nu_n
\geq \nu_n^0$, where
$$
\nu^0_n=
    \frac{1}{n}\sum_{j=0}^{n-1} f_*^j (\nu|{\tilde W_j}).
$$
Moreover, since $d \;f^j_*(\nu \mid \ti W_j)/d\nu$ is H\"older
continuous, the bounded distortion at Lemma~\ref{lem. bounded
distortion} implies that it is bounded from below by its $L^1$ norm
up to the multiplicative constant $K_0^{-1}$. So,
$$
\frac{d\nu^0_n}{d\nu}(y)
        = \frac{1}{n} \sum_{j=0}^{n-1} \frac{d \;f^j_*(\nu \mid \ti W_j)}{d\nu}(y)
        = \frac{1}{n} \sum_{j=0}^{n-1} \Big[ \sum_{\substack{ f^j(z)=y\\ z \in \ti W_j}}
                \la^{-j} e^{S_j \phi(z)}\Big]
        \geq K_0^{-1} \frac{1}{n} \sum_{j=0}^{n-1} \nu(\tilde W_j)
$$
for every $y \in D(x)$. Furthermore, by construction the set
$W_{j,\vep} \cap f^{-j}(D(x))$ is contained in $\tilde W_j$. This
guarantees that
$$
\frac{d\nu^0_n}{d\nu}(y)
    \geq K_0^{-1} \frac{1}{n} \sum_{j=0}^{n-1} \nu(\tilde W_j)
    \geq K_0^{-1} \nu^\vep_n(D(x))
    \geq K_0^{-1} \frac{\nu^\vep_\infty(D(x))}{2}
$$
for every large $n \geq 1$ in the subsequence of $(n_k)_k$ chosen
above. By weak$^*$ convergence it holds that $d\mu/d\nu \geq C_1$ in
the disk $D(x)$.
\end{proof}

We finish this section by proving the uniqueness of the equilibrium
state, which completes the proof of Proposition~\ref{prop. acim}.

\begin{lemma}\label{l. Gibbs property acim}
If $f$ is topologically mixing there is a unique invariant measure
$\mu$ absolutely continuous with respect to $\nu$. Moreover, the
density $d\mu/d\nu$ is bounded away from zero and infinity and the
sequences of hyperbolic times $\{n_j(x)\}$ are non-lacunary
$\mu$-almost everywhere. Furthermore, $\mu$ is a non-lacunary Gibbs
measure.
\end{lemma}

\begin{proof}
We have proven that any accumulation point $\mu$ of $(\nu_n)_n$ is
absolutely continuous with respect to $\nu$ and that the density
$h=d\mu/d\nu$ is bounded from above by $C_2$ and is bounded from
below by $C_1$ on some disk $D(x)$. Since $f$ is topologically
mixing there is $N \geq 1$ be such that $f^N(D(x))=M$, that is, any
point has some preimage by $f^N$ in $D(x)$. It is not difficult to
check that $h \in L^1(\nu)$ satisfies $\cL_\phi h=\la h$. Then
$$
h(y) = \la^{-N} \sum_{f^N(z)=y} e^{S_N\phi(z)} h(z)
    \geq C_1 \la^{-N} e^{N \inf\phi}
$$
for almost every $y \in M$, which allows to deduce that the measures
$\mu$ and $\nu$ are equivalent.

We claim that $\mu$ is ergodic. Indeed, if $G$ is any $f$-invariant
set such that $\mu(G)>0$ then it follows from Lemma~\ref{support
Gibbs} that there is a disk $\De$ of radius $\de/4$ such that
$\nu(\De \setminus G)=0$. Furthermore, using that $J_\nu f$ is
bounded from above and from below, the invariance of $G$ and that
there is $\tilde N \geq 1$ such that $f^{\tilde N}(\De)=M$ it
follows that $\nu(M \setminus G)=0$, or equivalently, that
$\mu(G)=1$, proving our claim. So, if $\mu_1 \ll \nu$ is any
$f$-invariant probability measure then $\mu_1 \ll \mu$. By
invariance of $d\mu_1/d\mu$ and ergodicity of $\mu$ it follows that
$d\mu_1/d\mu$ is almost everywhere constant and that $\mu_1=\mu$.
This proves the uniqueness of the absolutely continuous invariant
measure. Lemma~\ref{density control} also implies that
$$
C_3 \nu(B(x,n,\delta))
        \leq \mu(B(x,n,\delta))\leq
C_2\nu(B(x,n,\delta))
$$
for $\nu$-almost every $x$ and every $n \geq 1$, where $C_3=C_1
\la^{-N} e^{N \inf\phi}$. In particular $\mu$ is expanding and, if
$n$ is a hyperbolic time for $x$ and $y \in B(x,n,\de)$ then
$$
K^{-1}C_3
    \leq \frac{\mu(B(x,n,\delta))}{e^{-P n + S_{n}\phi(y)}}
\leq K C_2.
$$
Corollary~\ref{c.nu.expanding} implies that the first hyperbolic
time map $n_1$ is $\mu_i$-integrable. Hence, the sequence of
hyperbolic times is almost everywhere non-lacunary (see
Corollary~\ref{c.hyperbolicislacunary}) and both $\mu$ and $\nu$ are
non-lacunary Gibbs measures. This completes the proof of the lemma.
\end{proof}


\section{Proof of Theorems~\ref{Thm. Equilibrium States}
and~\ref{Thm. Equilibrium States2}}\label{Proof Thm 1}

In this section we manage to estimate the topological entropy of $f$
for the potential $\phi$ using the characterizations of relative
pressure given in Section~\ref{Relative Pressure}: $P_{H^c}(f,\phi)<
\log \la$ and $P_H(f,\phi)\leq \log \la$. Then, using that the
measure theoretical pressure $P_\mu(f,\phi)=h_\mu(f)+\int \phi
\,d\mu$ of every absolutely continuous invariant measure given by
Proposition~\ref{prop. acim} is at least $\log \la$, we deduce that
$\Ptop(f,\phi) =\log \la$ and that equilibrium states do exist.
Finally, the variational property of equilibrium states yields that
they coincide with the absolutely continuous invariant measures.
This will complete the proofs of Theorems~\ref{Thm. Equilibrium
States} and ~\ref{Thm. Equilibrium States2}.

\subsection{Existence of equilibrium states}

We give two estimates on the relative pressure and deduce the
existence of equilibrium states for $f$ with respect to $\phi$.

\begin{proposition}\label{pressure in H^c}
$P_{H^c}(f,\phi) < \log \lambda.$
\end{proposition}

Since we deal with a potential $\phi$ whose oscillation is not very
large, the main point in the proof of Proposition~\ref{pressure in
H^c} is to control the relative pressure $h_{H^c}(f)$. The key idea
is that $h_{H^c}(f)$ can be bounded above using the maximal
distortion and growth rate of the inverse branches that cover $H^c$.
We will begin with some preparatory lemmas.

\begin{lemma}\label{l.covers}
Let $M$ be a compact Besicovitch metric space of dimension $m$.
There exists $C>0$ and a sequence of finite open coverings
$(\cQ_k)_{k\geq 1}$ of $M$ such that $\diam(\cQ_k) \to 0$ as $k \to
\infty$, and every set $E\subset M$ satisfying $\diam(E)\leq D\diam
\cQ_k$ intersects at most $C D^m$ elements of $\cQ_k$.
\end{lemma}

\begin{proof}
First we construct a special family $\cT_k$ of partitions in $M$.
Let $(r_k)$ be a decreasing sequence of positive numbers converging
to zero. Given $k \geq 1$, let $X_k$ be a maximal $r_k$ separated
set: any two balls of radius $r_k$ centered at distinct points in
$X_k$ are pairwise disjoint and $X_k$ is a maximal set with this
property. In particular, it follows that $\{B(x,2r_k): x\in X_k\}$
is a covering of $M$. Since there exists no covering of $M$ by a
smaller number of balls as above, by Besicovitch covering lemma
there exists a constant $C_1$ (depending only on the dimension $m$)
that any point in $M$ is contained in at most $C_1$ balls. Consider
a partition $\cT_k$ in $M$ such that every element $T_k \in\cT_k$
contains a ball of radius $r_k$ and such that $\diam (\cT_k) \leq
2r_k$.

Fix a sequence of positive numbers $(\vep_k)_{k \geq 1}$ such that
$0< \vep_k \ll r_k$ for every $k \geq 1$. We claim that the family
$\cQ_k$ of open neighborhoods of size $\vep_k$ around elements of
$\cT_k$ satisfies the requirements of the lemma. It is immediate
that $\diam(\cQ_k) \to 0$ as $k \to \infty$. Since, by construction,
every point in $M$ is contained in at most $C_1$ elements of
$\cT_k$, any set $E\subset M$ satisfying $\diam(E)< D \diam(\cQ_k)
\leq 2 D\, (r_k +\vep_k)$ can intersect at most $[ 2 C_1 \,D \,(1+
\vep_k/ r_k))]^m$ elements of $\cQ_k$. This shows that $E$ can
intersect at most $C D^m$ elements of $\cQ_k$ for some constant $C$
depending only on the dimension $m$, completing the proof of the
lemma.
\end{proof}

The next result is the most technical lemma in the article and
provides the key estimate to prove Proposition~\ref{pressure in
H^c}.

\begin{lemma}\label{l.r.entropy.iteration}
Given any $\ell \geq 1$ the following property holds:
$$
h_{H^c}(f^\ell)\leq (\log q+m\log L+\vep_0/2)\,\ell+\log C.
$$
\end{lemma}

\begin{proof}
Fix $\ell \geq 1$ and let $(\cQ_k)_k$ be the family of finite open
coverings given by the previous lemma. Since $\diam(\cQ_k) \to 0$ as
$k \to \infty$ then
$$
P_{H^c}(f,\phi)
    =\lim_{k \to \infty} P_{H^c}(f,\phi,\cQ_k),
$$
by Definition~\ref{def. relative pressure1}. Recall $\cP$ is the
finite covering given by $(H2)$ and $B(n,\ga)$ is the set of points
whose frequency of visits to $\cA$ up to time $n$ is at least $\ga$.
The starting point is the next observation: \vspace{.1cm}

\noindent{\bf Claim 1:} \textit{For every $0<\vep<\gamma$ there
exists $j_0 \geq 1$ such that for every $j \geq j_0$ the following
holds:
$$
B(n,\ga) \subset B(\ell j,\ga-\vep)
    \quad \text{for every $ j \,\ell \leq n < (j+1) \ell$}.
$$}

\begin{proof}[Proof of Claim 1:]
Given $\vep>0$, let $j_0$ be a positive integer larger than
$(1-\ga)/\vep$. Given an arbitrary large $n$ we can write $n=\ell j
+ r$, where $0\leq r<\ell$ and $j \geq j_0$. Moreover, if $x$
belongs to $B(n,\ga)$ then $\#\,\{0\leq i \leq n-1: f^i(x) \in \cA\}
\geq \ga \,n$ and consequently
$$
\frac1{\ell j}\,\#\,\{0\leq i \leq \ell j-1: f^i(x) \in \cA\}
    \geq \ga +\frac{\ga r-r}{\ell j}.
$$
Our choice of $j_0$ implies that the right hand side above is
bounded from below by $\ga-\vep$. This shows that $x$ belongs to
$B(\ell j,\ga-\vep)$ and proves the claim.
\end{proof}

We proceed with the proof of the lemma. Observe that the set $H^c$
is covered by
$$
\bigcup_{n \geq N} \bigcup_{P \in \cP^{(n)}}
    \left\{P \in \cP^{(n)} : P \cap B(n,\ga) \neq \emptyset \right\}
$$
for every $N \geq 1$. Let $\vep>0$ be small such that $\#
I(n,\ga-\vep) \leq \exp (\log q + \vep_0/2)n$ for every large $n$.
Such an $\vep$ do exist because $c_\ga$ varies monotonically on
$\ga$ (see the proof of Lemma~\ref{l.combinatorio}). Then, the
previous claim allow us to cover $H^c$ using only cylinders whose
depth is a multiple of $\ell$: for any $N \geq 1$
\begin{equation}\label{eq.cover.H^c}
H^c \subset
 \bigcup_{j \geq \frac N\ell}
 \bigcup_{P \in \cP^{(\ell j)}}
        \left\{P \in \cP^{(\ell j)} : P \cap B(\ell j,\ga-\vep) \neq \emptyset
        \right\}.
\end{equation}
Thus, from this moment on we will only consider iterates $n =
j\,\ell$. Denote by $\cR^{(n)}$ the collection of cylinders in
$\cP^{(n)}$ that intersect $B(n,\ga-\vep)$. Our aim is now to cover
any element in $\cR^{(n)}$ by cylinders relatively to the
transformation $f^\ell$. Given $k \geq 1$, denote by $\cS_{f^\ell,j}
\cQ_k$ the set of $j$-cylinders of $f^\ell$ by elements in $\cQ_k$,
that is
$$
\cS_{f^\ell,j} \cQ_k
  =\Big\{ Q_0 \cap f^{-\ell}(Q_1) \cap \dots \cap
  f^{-\ell(j-1)}(Q_{j-1}) : Q_i \in \cQ_k, i=0,\dots,j-1 \Big\}.
$$
Furthermore, let $\mathcal G_{n,k}$ be the set of cylinders in
$\cS_{f^\ell,j} \cQ_k$ that intersect any element of $\cR^{(n)}$.
\vspace{0.1cm}

\noindent{\bf Claim 2:} \textit{Let $k \geq 1$ be large and fixed.
Then $$\# \mathcal G_{j \ell,k} \leq \#\,\cQ_k \, \times[C L^{\ell
m}]^j \times e^{(\log q + \vep_0/2)j \ell}$$ for every large $j$.}

\begin{proof}[Proof of Claim 2:]
Recall $n=j \ell$ and fix $P_{n} \in \cR^{(n)}$. Since $f$ is a
local homeomorphism then the inverse branch $f^{-n}: f^n(P_n) \to
P_n$ extends to the union of all $ Q\in \cQ_k$ so that $Q \cap
f^n(P_n) \neq \emptyset$, provided that $k$ is large. Notice that
$\diam(f^{-\ell}(Q)) \leq L^\ell \diam(Q)$ for every $Q \in \cQ_k$
because $\log \|Df(x)^{-1}\| \leq L$ for every $x \in M$. By
Lemma~\ref{l.covers}, $f^{-\ell}(Q)$ intersects at most $C L^{\ell
m}$ elements of the covering $\cQ_k$. This proves that there are at
most $\#\,\cQ_k \, \times[C L^{\ell m}]^j$ cylinders in
$\cS_{f^\ell,j}\cQ_k$ that intersect $P_n$. The claim is a direct
consequence of our choice of $\vep$ since $\# \cR^{(n)} \leq
e^{(\log q+ \vep_0/2)n}$ for large $n$.
\end{proof}

Finally we complete the proof of the lemma. Indeed, it is immediate
from \eqref{eq.cover.H^c} that
$$
m_\alpha(f^\ell,0,H^c,\cQ_k,N)
    \leq \sum_{j \geq N/\ell}
             \sum_{\mathrm{\underline{U}} \in \,\cG_{\ell j,k}}
             e^{-\alpha n(\mathrm{\underline{U}})}
    =  \sum_{j \geq N/\ell}
             e^{-\alpha j} \; \# \,\cG_{\ell j,k}
$$
for every large $k$. Moreover, Claim~2 implies that the sum in the
right hand side above converges to zero as $N \to \infty$
(independently of $k$) whenever $\al
>(\log q+\vep_0/2+m\log L) \,\ell +\log C$. This shows that
$h_{H^c}(f^\ell) \leq (\log q+m\log L+\vep_0/2) \,\ell +\log C$ and
completes the proof of the lemma.
\end{proof}

\begin{proof}[Proof of Proposition~\ref{pressure in H^c}] Recall
that $h_{H^c}(f^\ell)=\ell\,h_{H^c}(f)$, by
Proposition~\ref{p.pressure.iteration}. Then, as a consequence of
the previous lemma we obtain
$$
h_{H^c}(f)
  \leq \log q+m\log L+\vep_0/2 +\frac{\log C}{\ell}
$$
for every $\ell \geq 1$.
Finally, it follows from \eqref{eq. relation potential} and
Lemma~\ref{l.reduction} that
$$
P_{H^c}(f,\phi)
  \leq \log q+m\log L+\sup \phi+\vep_0
  < \log \deg f+\inf\phi
  \leq \log\la.
$$
\end{proof}

In the present lemma we give an upper bound on the relative pressure
of $\phi$ relative to the set $H$. More precisely,

\begin{lemma}\label{pressure in H}
$P_H(f,\phi)\leq \log \la.$
\end{lemma}

\begin{proof}
Recall the characterization of relative pressure using dynamical
balls in Subsection~\ref{Relative Pressure}. Pick $\alpha>\log \la$.
For any given $N \geq 1$, $H$ is contained in the union of the sets
$H_n$ over $n \geq N$. Thus, given $0<\vep \leq \delta$
$$
H \subset
    \bigcup_{n \geq N} \bigcup_{x\in H_n} B(x,n,\vep).
$$
Now we claim that there exists $D>0$ (depending only on $m=\dim(M)$)
so that for every $n\geq N$ there is a family $\cG_n \subset H_n$ in
such a way that every point in $H_n$ is covered by at most $D$
dynamical balls $B(x,n,\vep)$ with $x \in \cG_n$. In fact,
Besicovitch's covering lemma asserts that there is a constant $D>0$
(depending on $m$) and an at most countable family $\cG_n \subset
H_n$ such that every point of $f^n(H_n)$ is contained in at most $D$
elements of the family $\{B(f^n(x),\vep): x \in \cG_n\}$. Using that
each dynamical ball $B(x,n,\vep)$, $x\in H_n$, is mapped
homeomorphically onto $B(f^n(x),\vep)$, it follows that every point
in $H_n$ is contained in at most $D$ dynamical balls $B(x,n,\vep)$
with $x\in \cG_n$, proving our claim.
Given any positive integer $N \geq 1$, it follows by bounded
distortion and the Gibbs property of $\nu$ at hyperbolic times that
$$
m_\al(f,\phi,H,\vep,N)
    \leq K(\vep) \sum_{n\geq N} e^{-(\alpha-P)n}
    \Big\{\sum_{{x} \in \mathcal{G}_n}
    \nu(B(x,n,\vep))\Big\}.
$$
Consequently $m_\al(f,\phi,H,\vep,N) \leq K(\vep)\,
\frac{D}{1-e^{-(\alpha-P)}} \, e^{-(\alpha-P)N}$, which tends to
zero as $N \to \infty$ independently of $\vep$. This shows that
$P_H(f,\phi)\leq \log \la$ and completes the proof of the lemma.
\end{proof}

We know that every ergodic component of an absolutely continuous
invariant measure is also absolutely continuous. Now we prove that
the absolutely continuous invariant measures are indeed an
equilibrium states.

\begin{lemma}\label{Gibbs imply equilibrium}
If $\mu$ is an ergodic measure absolutely continuous with respect to
$\nu$ then $P_\mu(f,\phi)\geq \log \la$. Moreover, $\mu$ is an
equilibrium state for $f$ with respect to $\phi$ and the following
equalities hold
$$
\Ptop(f,\phi)=P_H(f,\phi)=\log \la.
$$
\end{lemma}

\begin{proof}
The previous estimates and Proposition~\ref{non-compact variational
principle} guarantee that
$$
\Ptop(f,\phi)
    =\sup\{P_{H}(f,\phi), P_{H^c}(f,\phi)\} \leq \log \la.
$$
Using that $d\mu/d\nu \leq C_2$, that $\nu$ satisfies the Gibbs
property at hyperbolic times and $\mu$-almost every point $x$ admits
a sequence $\{n_k(x)\}$ of hyperbolic times then
\begin{equation*}
\mu(B(x,n_k,\vep))
    \leq C_2 \,K(\vep) \,e^{-P n_k + S_{n_k} \phi(y)}
\end{equation*}
for every $0<\vep \leq \de$, every $k\geq 1$ and every $y \in
B(x,n_k,\vep )$. Thus, Brin-Katok's local entropy formula for
ergodic measures and Birkhoff's ergodic theorem (see e.g.
\cite{Man87}) immediately imply that
$$
h_\mu(f)
    =\lim_{\vep\to 0}\limsup_{n \to \infty} -\frac{1}{n} \log \mu(B(x,n,\vep))
    \geq P - \int \phi \,d\mu,
$$
where the first equality holds $\mu$-almost everywhere. In
particular
$$
\log \la
    \geq \Ptop(f,\phi)
    \geq P_{H}(f,\phi)
    \geq \sup_{\mu(H)=1} \Big\{h_{\mu}(f)+\int \phi \,d\mu\Big\}
    \geq \log \la,
$$
which proves that $\mu$ is an equilibrium state and that the three
quantities in the statement of the lemma do coincide. This completes
the proof of the lemma.
\end{proof}

\subsection{Finitude of ergodic equilibrium states}

In this subsection we will complete the proof of Theorems~\ref{Thm.
Equilibrium States} and ~\ref{Thm. Equilibrium States2} and
Corollary~\ref{Cor.ContinuousPotentials}. First we combine that
every equilibrium state is an expanding measure with some ideas
involved in the proof of the variational properties of SRB measures
in \cite{Le84a} to deduce that every equilibrium state is absolutely
continuous with respect to some conformal measure supported in the
closure of the set $H$, and to obtain finitude of ergodic
equilibrium states. Finally, we show that under the topologically
mixing assumption there is a unique equilibrium state, and that it
is exact and a non-lacunary Gibbs measure. We begin with the
following abstract result:

\begin{theorem}\label{thm.equilibrium is acim}
Let $f:M \to M$ be a local homeomorphism, $\phi:M \to \R$ be a
H\"older continuous potential and $\nu$ be a conformal measure such
that $J_\nu f=\la e^{-\phi}$, where $\la=\exp (\Ptop(f,\phi))$.
Assume that $\eta$ is an equilibrium state for $f$ with respect to
$\phi$ gives full weight to $\supp (\nu)$ and that
$$
\lim_{n\to\infty} \frac1n \sum_{j=0}^{n-1} L(f^j(x)) <0
$$
almost everywhere. Then $\eta$ is absolutely continuous with respect
to $\nu$.
\end{theorem}

Let us stress out that this theorem holds in a more general setting.
Since this fact will not be used here, we will postpone the
discussion to Remark~\ref{rmk.general.endomorphism} near the end of
the section. The finitude of equilibrium states is a direct
consequence of the previous result. Indeed,

\begin{corollary}\label{c.equilibrium is acim}
Let $f$ be a local homeomorphism and let $\phi$ be a H\"older
continuous potential satisfying (H1), (H2) and (P). There exists an
expanding conformal probability measure $\nu$ such that every
equilibrium state for $f$ with respect to $\phi$ is absolutely
continuous with respect to $\nu$ with density bounded from above.
If, in addition, $f$ is topologically mixing then there is unique
equilibrium state and it is a non-lacunary Gibbs measure.
\end{corollary}

\begin{proof}
Let $\nu$ be the expanding conformal measure given by
Theorem~\ref{thm. conformal measure} and $\eta$ be an ergodic
equilibrium state for $f$ with respect to $\phi$. We claim that
$\eta$ is an expanding measure. Indeed, assume by contradiction that
one can decompose $\eta$ as a linear convex combination of two
measures $\eta=t\eta_1+(1-t)\eta_2$ with $\eta_2(H^c)=1$ for some
$0\leq t<1$. But Lemma~\ref{Gibbs imply equilibrium}, the first part
of Proposition~\ref{non-compact variational principle} and the
convexity of the pressure yield
$$
P_\eta(f,\phi)
    =t P_{\eta_1}(f,\phi)+(1-t) P_{\eta_2}(f,\phi)
    \leq t \,\Ptop(f,\phi)+(1-t)\,P_{H^c}(f,\phi)
    <\Ptop(f,\phi),
$$
which contradicts that $\eta$ is an equilibrium state and proves our
claim. Moreover, $\eta(\supp (\nu))=1$ because the support of $\nu$
coincides with the closure of $H$. Finally, since
$$
\limsup_{n\to \infty} \frac 1n \sum_{j=0}^{n-1}
                        \log L(f^j(x))\leq -2c<0
$$
at $\eta$-almost every point (Corollary~\ref{c.equilibrium is
acim}), the assumptions of Theorem~\ref{thm.equilibrium is acim} are
verified. This result is a direct consequence of the previous
theorem and Proposition~\ref{prop. acim}.
\end{proof}

In the sequel we prove Theorem~\ref{thm.equilibrium is acim}. Since
$f$ is a non-invertible transformation we use the natural extension,
introduced in Subsection \ref{Natural extension}, to deal with
unstable manifolds.

\begin{proof}[Proof of Theorem~\ref{thm.equilibrium is acim}]
It is easy to check, using the variational principle, that almost
every ergodic component of an equilibrium state is an equilibrium
state. Thus, by ergodic decomposition it is enough to prove the
result for ergodic measures.

Let $\eta$ be an ergodic equilibrium state and $(\hat f, \hat\eta)$
be the natural extension of $\eta$ introduced in
Subsection~\ref{Natural extension}. Then ${\pi}_* \hat\eta=\eta$ and
that
$$
 \lim_{n\to \infty} \frac 1n \sum_{j=0}^{n-1} \log \hat L(\hat f^j(\hat x))<0
$$
$\hat\eta$-almost everywhere.

We proceed with the construction of a special partition $\hat\cQ$ of
$\hat M$ that is closely related with Ledrappier's geometric
construction in Proposition~3.1 of \cite{Le84a} and provides a key
ingredient for the proof of Theorem~\ref{thm.equilibrium is acim}.
The main differences from the original result due to Ledrappier are
that the natural extension $\hat M$ is not in general a manifold and
that there is no well defined unstable foliation in $M$. Given a
partition $\hat\cQ$ denote by $\hat\cQ(\hat x)$ the element of
$\hat\cQ$ that contains $\hat x \in \hat M$. We say that $\hat Q$ is
an \emph{increasing partition} if $(\hat f^{-1}\hat \cQ)(\hat x)
\subset \hat\cQ(\hat x)$ for $\hat\eta$-almost every $\hat x$, in
which case we write $\hat f^{-1}\hat \cQ\succ \hat\cQ$.

\begin{proposition}\label{p.generating.partition}
There exists an invariant and full $\hat\eta$-measure set $\hat S
\subset \hat M$, and a measurable partition $\hat\cQ$ of $\hat S$
such that:
\begin{enumerate}
\item\label{ep1} $\hat f^{-1}\hat \cQ\succ \hat\cQ$,
\item\label{ep2} $\bigvee_{j=0}^{+ \infty} \hat f^{-j} \hat\cQ$ is the
        partition into points,
\item\label{ep4} The sigma-algebras $\cM_n$ generated by the partitions
        $\hat f^{-n} \hat\cQ$, $n\geq 1$, generate the
        $\si$-algebra in $\hat S$, \, and
\item\label{ep3} For almost every $\hat x$ the element
        $\hat\cQ(\hat x)\subset \hat W^u(\hat x)$ contains a neighborhood of $\hat x$ in
        $\hat W^u(\hat x)$ and the  projection $\pi(\hat\cQ(\hat x))$ contains
        a neighborhood of $x_0$ in $M$.
\end{enumerate}
\end{proposition}

\begin{proof}
Since $\hat\eta$ is an expanding measure,
Proposition~\ref{p.Pesinpointwise} guarantees the existence of local
unstable manifolds at $\hat\eta$-almost every point. Take $i \geq 1$
such that $\hat\eta(\hat\La_i)>0$ and let $r_i$, $\vep_i$, $\ga_i$
and $R_i$ be given by Corollary~\ref{c.Pesinblocks}. Fix also
$0<r\leq r_i$ and $\hat x \in \supp(\hat\eta\mid_{\hat\La_i})$.
Recall that $\hat y \mapsto \Wloc^u(\hat y) \cap B(x_0,r)$ is a
continuous function on $B(\hat x, \vep_i r) \cap \hat\La_i$.
Consider the sets
$$
\hat V(\hat y,r)=\{\hat z \in \hWloc^u(\hat y) : z_0 \in B(x_0,r)\},
$$
defined for any $\hat y \in B(\hat x, \vep_i r) \cap \hat\La_i$.
Define also
$$
\hat S (\hat x,r)
    =\bigcup \{\hat V (\hat y, r) : \hat y \in B(\hat x,\vep_i r) \cap \hat\La_i\}
$$
and the partition $\hat\cQ_0(r)$ of $\hat M$ whose elements are the
connected components $\hat V(\hat y,r)$ of unstable manifolds just
constructed and their complement $\hat M \setminus \hat S (\hat
x,r)$. Furthermore, consider the set $\hat S_r$ and the partition
$\hat \cQ(r)$ given by
\begin{equation*}
\hat S_r = \bigcup_{n=0}^{+\infty} \hat f^n(\hat S(\hat x,r))
    \quad\text{and}\quad
\hat Q(r) = \bigvee_{n=0}^{+\infty} \hat f^n(\hat Q_0(r)).
\end{equation*}
Then, the partition $\hat\cQ$ coincides with the partition
$\hat\cQ(r)$ and the set $\hat S$ is given by $\bigcap_{j\geq 0}
\hat f^{-j}(\hat S_r)$ for a particular choice of the parameter $r$.
In what follows, for notational convenience and when no confusion is
possible we shall omit the dependence of the partition $\hat\cQ$ on
$r$.

It is clear from the construction that every partition $\hat\cQ(r)$
is increasing, that is the content of \eqref{ep1}. In addition,
since $\hat\eta$ is ergodic and $\hat\eta(\hat S(\hat x,r))>0$ then
the set of points that return infinitely often to $\hat S(\hat
x,r)$, which we called $\hat S_r$, is a full measure set by
Poincar\'e's Recurrence Theorem. In other words, if a point $\hat y$
belongs to $\hat S_r$ there are positive integers $(n_j)_{j}$ such
that $\hat f^{n_j}(\hat y) \in \hat V(\hat f^{n_j}(\hat y),r)$.
Hence, the backward distance contraction along unstable leaves
guarantees that the diameter of the partition $\bigvee_{n=0}^{n}
\hat f^{-j}\hat\cQ$ tend to zero as $n \to\infty$, proving
\eqref{ep2}.
By construction, there is a full measure set such that any two
distinct points $\hat y$ and $\hat z$ lie in different elements of
$\hat f^{-n}\hat\cQ$ for some $n \in \N$. Indeed, if $\hat
f^{-n}\hat\cQ(\hat y) = \hat f^{-n}\hat\cQ(\hat z)$ for every $n
\geq 0$ then $\hat f^n(\hat y)$ and $\hat f^n(\hat z)$ lie
infinitely often in the same local unstable manifold. But
\eqref{ep2} implies that $\hat y$ and $\hat z$ should coincide,
which is a contradiction and proves our claim. In particular, the
decreasing family of $\si$-algebras $\cM_n$, $n \geq 1$, generate
the $\si$-algebra in $\hat S_r$, which proves \eqref{ep4}.

We proceed to show that the partition $\hat\cQ(r)$ satisfies
\eqref{ep3} for Lebesgue almost every parameter $r$. Given  $0<r\leq
r_i$ and $\hat y \in \hat S_r$ define
$$
\beta_r(\hat y)
    =\inf_{n \geq 0}\Big\{R_i, \frac{r}{\ga_i}, \frac{1}{2\ga_i} e^{\la_i n} d(y_{-n}, \partial
    B(x_0,r))\Big\},
$$
that it clearly non-negative. It is enough to obtain the following:
\begin{itemize}
\item[(a)] If $z_0 \in \Wloc^u(\hat y)$ and $d(y_0,z_0)<\beta_r(\hat
            y)$ then there exists $\hat z \in \hat\cQ(\hat y)$ such that
            $\pi(\hat z)=z_0$;
\item[(b)] There exists a full Lebesgue measure set of parameters $0<r\leq r_i$
            such that the function $\beta_r(\cdot)$ is
            strictly positive almost everywhere and $\hat\eta(\partial \hat\cQ(r))=0$.
\end{itemize}

Take $\hat y \in \hat S_r$ and assume that $z_0 \in \Wloc^u(\hat y)$
is such that $d(y_0,z_0)<\beta_r(\hat y)$. If $\hat y \in \hat
S(\hat x,r)$ then there exists $\hat w \in B(\hat x,\vep_i r)$ such
that $\hat y \in \hWloc^u(\hat w)$. Furthermore, since
$d(y_0,z_0)<\beta_r(\hat y)<R_i$ then there exists $\hat z \in
\hWloc^u(\hat w)$ such that $\pi(\hat z)=z_0$. Hence
$$
d(y_{-n},z_{-n}) \leq \ga_i e^{-n \la_i} d(y_0,z_0), \;\forall n \in
\N,
$$
which implies that $d(y_{-n},z_{-n}) \leq r$ and $d(y_{-n},z_{-n})
\leq 1/2 \,d(y_{-n},\partial B(x_0,r))$ for every $n \in \N$.
Together with Corollary~\ref{c.Pesinblocks}, this shows that
$y_{-n}$ and $z_{-n}$ belong to the same element of the partition
$\hat\cQ_0$ for every $n \geq 1$ and, assuming (b) for the moment,
that $\pi(\cQ(\hat y))$ contains a neighborhood of $y_0$ in
$\Wloc^u(\hat y)$. On the other hand, if $\hat y \in \hat S_r
\setminus \hat S(\hat x,r)$ then there exists $k \geq 1$ such that
$\hat f^{-k}(\hat y) \in \hat S(\hat x,r)$ and consequently the
projection of the set
$$
\hat\cQ(\hat y)=\hat f^{k} (\hat\cQ( \hat f^{-k}(\hat y))
$$
contains an open neighborhood of $y_0$ in $\Wloc^u(\hat y)$. This
completes the proof of (a).

The proof of (b) is slightly more involving. We begin with the
following remark from measure theory: if $r_0>0$, $\vartheta$ is a
Borel measure in $[0,r_0]$ and $0<a<1$ then Lebesgue almost every $r
\in[0,r_0]$ satisfies
\begin{equation}\label{eq.sum}
\sum_{k=0}^{\infty} \vartheta \big( [r-a^k, r+a^k] \big)< \infty.
\end{equation}
Indeed, the set
$$
B_{a,k}
    = \Big\{r \in [0,r_0]: \vartheta \big( [r-a^k, r+a^k] \big)> \frac{\vartheta \big( [0, r_0] \big)}{k^2} \Big\}
$$
can be covered by a family $I_k$ of balls of radius $a^k$ centered
at points of $B_{a,k}$ in such a way that any point is contained in
at most two intervals of $I_k$. Since
$$
\# I_k \,\frac{\vartheta \big( [0, r_0] \big)}{k^2}
    \leq \sum_{I \in I_k} \vartheta(I)
    \leq 2 \,\vartheta([0,r_0])
$$
then $\#I_k \leq 2 k^2$ and it is clear that $ \Leb(B_{a,k}) \leq 2
a^k \# I_k$ is summable. Borel-Cantelli's lemma implies that
Lebesgue almost every $r \in[0,r_0]$ belongs to finitely many sets
$B_{a,k}$, which proves the summability condition in \eqref{eq.sum}.

Back to the proof of (b), let $\vartheta$ be the measure of the
interval $[0,r_i]$ defined by $\vartheta( E ) = \hat\eta \big(\,
\hat y \in \hat M: d(x_0,y_0) \in E \,\big)$. The previous assertion
guarantees that for Lebesgue almost every $r \in[0,r_i]$ it holds
\begin{equation}\label{eq.sum2}
\sum_{k=0}^{\infty} \hat\eta \big(\, \hat y \in \hat M:
|d(x_0,y_0)-r|< e^{-\la_i k} \,\big) < \infty.
\end{equation}
On the other hand, there exists $D>0$ such that
$|d(z_0,x_0)-r|<D\tau$ whenever $d(z_0,\partial B(x_0,r))<\tau$ and
$0<\tau<r\leq r_i$. Therefore
\begin{multline*}
\sum_{k=0}^{\infty} \hat\eta \big(\, \hat y \in \hat M: |d(y_{-n},
\partial B(x_0,r))| < D^{-1}e^{-\la_i k} \,\big)
    \leq \\
\leq \sum_{k=0}^{\infty} \hat\eta \big(\, \hat y \in \hat M:
|d(x_0,y_{-n})-r|< e^{-\la_i k} \,\big),
\end{multline*}
which is finite because of the invariance of $\hat\eta$ and the
former condition \eqref{eq.sum2}. Using Borel-Cantelli's lemma once
more it follows that $\hat\eta$-almost every $\hat y$ satisfies
$$
|d(y_{-n},\partial B(x_0,r))| < D^{-1}e^{-\la_i k}
$$
for at most finitely many positive integers $k$, proving that
$\beta_r(\hat y)>0$. Furthermore, since $\eta(\cup_{n\geq 0} f^n
(\partial B(x_0,r)))=0$ for all but a countable set of parameters
$0<r \leq r_i$ then $\hat\cQ(\hat y)$ contains a neighborhood of
$\hat y$ in $\hWloc^u(\hat y)$ for $\hat\eta$-almost every $\hat
y\in \hat M$. This shows that (b) holds and, in consequence, for
Lebesgue almost every $r\in[0,r_i]$ the partition $\hat\cQ(r)$
satisfies the requirements of the proposition.
\end{proof}

Let $(\hat\eta_x)_x$ be the disintegration of the measure $\hat\eta$
on the measurable partition $\hat\cQ$, given by Rokhlin's theorem.
Recall that for $\hat\eta$-almost every $\hat x$ the map
$\pi\mid_{\hWloc^u(\hat x)} : \hWloc^u(\hat x) \to \Wloc^u(\hat x)$
is a bijection. For any such $\hat x$ let $\hat\nu_x$ be the measure
on $\hWloc^u(\hat x)$ obtained as the pull-back of
$\nu\mid_{\Wloc^u(\hat x)}$ by the bijection $\pi\mid_{\hWloc^u(\hat
x)}$. Let $\hat\nu$ denote the measure defined on $\hat M$ by the
disintegration $(\hat\nu_{\hat x})_{\hat x}$, that is to say that
$$
\hat \nu(\hat E)
    = \int \hat\nu_{\hat x} (\hat E) \, d\hat\eta (\hat x)
$$
for every measurable set $\hat E$ in $\hat M$. As a byproduct of the
previous result we obtain

\begin{corollary}\label{c.positivemeasure}
$0<\hat\nu_{\hat x}(\hat\cQ(\hat x))<\infty$, for $\hat\eta$-almost
every $\hat x$.
\end{corollary}

\begin{proof}
For every $\hat x$ in a full $\hat\eta$-measure set one has that
$$
\hat\nu_{\hat x}(\hat\cQ(\hat x))
    =\nu\big( \pi(\hat\cQ(\hat x))\cap \Wloc^u(\hat x) \big).
$$
Since $\hat\eta$ is an expanding measure then $\hWloc^u(\hat x)$ is
a neighborhood $\hat x$ and $\Wloc^u(\hat x) \cap \pi(\hat\cQ(\hat
x))$ contains a neighborhood of $x_0$ in $M$. In addition, since
$\eta(\supp \nu)=1$, for every $\hat x$ in a full $\hat\eta$-measure
set it holds that $x_0 \in \supp(\nu)$. Then it is clear that
$0<\hat\nu_{\hat x}(\hat\cQ(\hat x))<\infty$, $\hat\eta$-almost
everywhere, which proves the corollary.
\end{proof}

The next preparatory lemma shows that $\hat\nu$ has a Jacobian with
respect to $\hat f$ and establishes Rokhlin's formula for the
natural extension.

\begin{lemma}\label{l.Pesin.extension}
The measure $\hat\nu$ has a Jacobian $J_{\hat \nu} \hat f = J_{\nu}
f \circ \pi$ with respect to $\hat f$. In addition,
$$
h_{\hat\eta}(\hat f)= \int \log J_{\hat \nu} \hat f \;d\hat\eta.
$$
Furthermore, for $\hat\eta$-almost every $\hat x$ and every $\hat y
\in \hat\cQ(\hat x)$ the product
$$
\De(\hat x,\hat y) = \prod_{j=1}^{\infty }
        \frac{J_{\hat \nu} \hat f(\hat f^{-j}(\hat x))}{J_{\hat \nu} \hat f(\hat f^{-j}(\hat y))}
$$
is positive and finite.
\end{lemma}

\begin{proof}
Since the sigma-algebra $\hat\cB$ is the completion of the
sigma-algebra generated by the cylinders $\pi_i^{-1} (f^{-i} \cB)$,
$i \geq 1$, then the first claim in the lemma is a consequence from
the fact that
\begin{equation}\label{eq.extension.conformal0}
\hat\nu_{\hat f (\hat x)} (\hat f(\hat E))
    = \int_{\hat E \cap (\hat f^{-1} \hat\cQ)(\hat x)} J_\nu f \circ \pi \, d\hat\nu_{\hat x}
\end{equation}
for almost every $\hat x$ and every small cylinder $\hat E=
\pi^{-1}(E)$. Indeed, if $\hat E$ is a small cylinder then it is
clear that
\begin{equation}\label{eq.extension.conformal}
\hat\nu (\hat f (\hat E))
    =\int \hat\nu_{\hat f (\hat x)} (\hat f(\hat E)) \, d\hat\eta(\hat x)
    = \int \int_{\hat E \cap (\hat f^{-1} \hat\cQ)(\hat x)} J_\nu f \circ \pi \, d\hat\nu_{\hat
    x} \,d\hat\eta(\hat x).
\end{equation}
Let $\tilde\nu_{\hat x}$ denote the restriction of the measure
$\hat\nu_{\hat x}$ to the set $(\hat f^{-1}\hat\cQ)(\hat x) \subset
\hat\cQ(\hat x)$. Then $\hat\nu$ has a disintegration $\hat\nu=\int
\tilde\nu_{\hat x} \, d\hat\eta$ with respect to the measurable
partition $\hat f^{-1}\hat\cQ$. Together with
\eqref{eq.extension.conformal} this gives
\begin{equation*}
\hat\nu (\hat f (\hat E))
    = \int \int_{\hat E} J_\nu f \circ \pi \, d\tilde\nu_{\hat x} \,d\hat\eta(\hat x)
    = \int_{\hat E} J_\nu f \circ \pi \, d\hat\nu,
\end{equation*}
which proves that $\hat\nu$ has a Jacobian and $J_{\hat\nu}\hat f=
J_\nu f \circ \pi$. Hence, to prove the first assertion in the lemma
we are reduced to prove \eqref{eq.extension.conformal0} above. If $f
\mid E$ is injective and $\hat E = \pi^{-1}(E)$ then
\begin{multline*}
\hat\nu_{\hat f (\hat x)} (\hat f(\hat E))
    = \hat\nu_{\hat f (\hat x)}
        \big( \hat f [ \hat E \cap (\hat f^{-1}\hat\cQ)(\hat x)] \big)
    = \nu \big( f( E \cap \pi((\hat f^{-1}\hat\cQ)(\hat x)) \big)\\
    = \int_{E \cap \pi((\hat f^{-1}\hat\cQ)(\hat x)} J_\nu f \, d\nu
    = \int_{\hat E \cap (\hat f^{-1}\hat\cQ)(\hat x)} J_\nu f \circ \pi \, d\hat\nu_{\hat x},
\end{multline*}
which proves \eqref{eq.extension.conformal0}. On the other hand,
$h_{\eta}(f)= \int J_\nu f \, d\eta$ because $\eta$ is an
equilibrium state, $\Ptop(f,\phi)=\log \la$ and $J_\nu f = \la
e^{-\phi}$. So, using $\pi_*\hat\eta=\eta$ we obtain
$$
h_{\hat\eta} (\hat f)
        = h_\eta(f)
        = \int \log J_\nu f \,d\eta
        = \int \log (J_\nu f \circ\pi) \,d\hat\eta
        = \int \log J_{\hat \nu} \hat f \;d\hat\eta,
$$
which proves the second assertion in the lemma. Finally, the
H\"older continuity of the Jacobian $J_{\hat\nu}\hat f=J_{\nu}f
\,\circ \,\pi$, the fact that $\hat\cQ$ is subordinated to unstable
leaves and the backward distance contraction for points in the same
unstable leaf yield that the product
$$
\De(\hat x,\hat y) = \prod_{j=1}^{\infty }
                \frac{J_{\hat \nu} \hat f(\hat f^{-j}(\hat x))}{J_{\hat \nu} \hat f(\hat f^{-j}(\hat y))}
$$
is convergent for almost every $\hat x$ and every $\hat y \in
\hat\cQ(\hat x)$. The proof of the lemma is now complete.
\end{proof}

The last main ingredient to the proof of
Theorem~\ref{thm.equilibrium is acim} is the following generating
property of the partition $\hat\cQ$.

\begin{proposition}\label{p.entropygenerating}
$h_{\hat\eta}(\hat f)
    =H_{\hat\eta}( \hat f^{-1} \hat\cQ \mid \hat\cQ)$.
\end{proposition}

The proof of this result involves two preliminary lemmas. Let $i
\geq 1$ and $\hat\La_i$ be given as in the proof of
Proposition~\ref{p.generating.partition} and $r_i$ given by
Corollary~\ref{c.Pesinblocks}. The following lemma gives a dynamical
characterization of the local unstable manifolds.

\begin{lemma}\label{l.D.ep}
Given $\vep>0$ there is a measurable function $\hat D_\vep: \hat
B_\la \to \R_+$ satisfying $\log\hat D_\vep \in L^1(\hat\eta)$ and
such that, if $d(x_{-n},y_{-n}) \leq \hat D_\vep(\hat f^{-n}(\hat
x)) \; \forall n \geq 0$ then $\hat y \in \hWloc^u(\hat x)$ and
$d(x_0,y_0)<2r_i$.
\end{lemma}

\begin{proof}
Since $\hat\eta(\hat\La_i)>0$ and $\hat\eta$ is assumed to be
ergodic then some iterate of almost every point will eventually
belong to $\hat\La_i$ by Poincar\'e's recurrence theorem. So, the
first hitting time $R(\hat x)$ is well defined almost everywhere in
$\hat\La_i$ and $\int_{\hat\La_i} R \,d\hat\eta =
1/\hat\eta(\hat\Lambda_i)$, by Kac's lemma. This proves that the
logarithm of the function $\hat D_\vep: \hat M \to \R$ given by
$$
\hat D_\vep(\hat x)=
    \begin{cases}
    \begin{array}{ll}
    \min\big\{ 2r_i \,,\, \de_i \,,\, \de_i/\ga_i \big\} \, e^{-(\la+\vep) R(\hat x)}
                    & ,\text{if}\; \hat x \in \hat\La_i\\
    \min\big\{2r_i \,,\, \de_i \,,\, \de_i/\ga_i\big\}
                    & ,\text{otherwise}
    \end{array}
    \end{cases}
$$
is $\hat\eta$-integrable. On the other hand, if $\hat x
\in\hat\La_i$ then $R(\hat f^{-n}(\hat x))=n$. Any $\hat y \in \hat
M$ such that $ d(x_{-n},y_{-n})\leq \hat D_\vep(\hat f^{-n}(\hat
x))$ for every  $n \geq 0$ clearly satisfies $d(x_0,y_0)<2r_i$ and,
by Proposition~\ref{p.Pesinpointwise}(2), belongs to $\Wloc^u(\hat
x)$. This concludes the proof of the lemma.
\end{proof}

This result allow us to construct an auxiliary measurable partition
of finite entropy that will be useful to compute the metric entropy
$h_{\hat\eta}(\hat f)$.

\begin{lemma}\label{l.partition.Mane}
There exists a measurable partition $\hat \cP$ of $\hat S$ such that
$H_{\hat\eta}(\hat\cP)<\infty$, $\diam(\hat\cP(\hat x)) \leq \hat
D_\vep(\hat x)$ at $\hat\eta$-almost every $\hat x$, and that the
partition
$$
\hat \cP^{(\infty)}
    =\bigvee_{n=0}^{+\infty} \hat f^n \hat\cP
$$
is finer than $\hat\cQ$.
\end{lemma}

\begin{proof}
Let $\hat D_\vep$ be the measurable function given by the previous
lemma. By Lemma 2 in \cite{Man81}, there exists a measurable and
countable partition $\hat\cP_0$ such that
$H_{\hat\eta}(\hat\cP_0)<\infty$ and $\diam \hat\cP(\hat x) \leq
\hat D_\vep(\hat x)$ for a.e. $\hat x \in \hat M$. Let $\hat\cP$ be
the finite entropy partition obtained as the refinement of $\hat
\cP_0$ and $\{\hat M \setminus \hat S(\hat x,r), \hat S(\hat
x,r)\}$. Notice that there is a full measure set where any two
points $\hat x$ and $\hat y$ belong to the same element of $\hat f^n
\hat\cP$ for every $n \geq 0$ if and only
$$
d(x_{-n}, y_{-n}) \leq \hat D_\vep(\hat f^{-n} \hat x)
    \quad \text{for every} \; n \geq 0.
$$
In particular, Lemma~\ref{l.D.ep} above implies that each element of
$\hat\cP$ is a piece of some local unstable manifold. Hence, since
$\hat\cP$ was chosen to refine $\{\hat M \setminus \hat S(\hat x,r),
\hat S(\hat x,r)\}$ then it is easy to see that
$$
\bigcap_{n \geq 0} \hat f^n \hat\cP(\hat f^{-n}(\hat x))
    \subset \hat\cQ(\hat x).
$$
for almost every $\hat x$. So, the partition $\hat\cP$ just
constructed satisfies the conclusions of the lemma.
\end{proof}

\begin{proof}[Proof of Proposition~\ref{p.entropygenerating}]
Let $\vep>0$ be arbitrary small. Up to a refinement of the partition
$\hat\cP$ we may assume without loss of generality that
$h_{\hat\eta}(\hat f,\hat\cP) \geq h_{\hat\eta}(\hat f)-\vep$. Since
the partition $\hat\cP^{(\infty)}$ is finer than $\hat\cQ$ then
$$
h_{\hat\eta}(\hat f,\hat\cP)
        = h_{\hat\eta}(\hat f,\hat\cP^{(\infty)})
        = h_{\hat\eta}(\hat f,\hat\cP^{(\infty)} \vee \hat\cQ)
        = h_{\hat\eta}(\hat f,\hat f^n \hat\cP^{(\infty)} \vee \hat\cQ)
$$
for every $n \geq 1$. Using that $h_{\hat\eta}(\hat f,\hat
\xi)=H_{\hat\eta}(\hat f^{-1} \hat \xi,\hat\xi)$ for every
increasing partition $\hat\xi$, the right hand side term in the
previous equalities coincides with the relative entropy
 $H_{\hat\eta}( \hat f^n \hat\cP^{(\infty)} \vee \hat\cQ \mid \hat f^{n+1} \hat\cP^{(\infty)} \vee \hat f\hat\cQ)$
and, consequently,
\begin{equation*}
h_{\hat\eta}(\hat f, \hat\cP)
 =
 H_{\hat\eta}(\hat\cQ \mid \hat f\hat\cQ \vee \hat f^n \hat\cP^{(\infty)})
    +
 H_{\hat\eta}(\hat\cP^{(\infty)} \mid \hat f^{-n} \hat\cQ \vee \hat f \hat\cP^{(\infty)}).
\end{equation*}
The second term in the right hand side above is bounded by
$H_{\hat\eta}(\hat\cP)$, which is finite. Then
Proposition~\ref{p.generating.partition}\eqref{ep4} implies that it
tends to zero as $n\to\infty$. On the other hand, the diameter of
almost every element in $\hat f^{-n+1}\hat\cQ$ tend to zero as $n
\to\infty$, proving that there exists a sequence of sets $(\hat
D_n)_{n\geq 1}$ in $\hat M$ satisfying $\lim_n \hat\eta(\hat D_n)=1$
and such that $ \hat f\cQ(\hat x) \subset \hat f^n
\hat\cP^{(\infty)}(\hat x) \quad\text{for every}\; \hat x \in \hat
D_n$. Then
\begin{multline*}
 H_{\hat\eta}(\hat\cQ \mid \hat f\hat\cQ \vee \hat f^n \hat\cP^{(\infty)})
        = \int - \log \hat\eta_{(\hat f\hat\cQ \vee \hat f^n \hat\cP^{(\infty)})(\hat
        x)} (\hat\cQ(\hat x))\, d\hat\eta(\hat x) \geq \\
        \geq \int_{\hat D_n(\hat x)} - \log \hat\eta_{(\hat f\hat\cQ)(\hat
        x)} (\hat\cQ(\hat x)) \, d\hat\eta(\hat x),
\end{multline*}
where the measures $\hat\eta_{\hat f\hat\cQ \vee \hat f^n
\hat\cP^{(\infty)}}$ and $\hat\eta_{\hat f\hat\cQ}$ denote
respectively the conditional measures of $\eta$ with respect to the
partitions $\hat f\hat\cQ \vee \hat f^n \hat\cP^{(\infty)}$ and
$\hat f\hat\cQ$. This proves that $\lim_{n} H_{\hat\eta}(\hat\cQ
\mid \hat f\hat\cQ \vee \hat f^n \hat\cP^{(\infty)}) \geq
H_{\hat\eta}(\hat\cQ \mid \hat f\hat\cQ)$. Since the other
inequality is always true we deduce that $h_{\hat\eta}(\hat f,
\hat\cP)=H_{\hat\eta}(\hat\cQ \mid \hat f\hat\cQ)$. Since $\vep>0$
was chosen arbitrary this proves that $h_{\hat\eta}(\hat
f)=H_{\hat\eta}(\hat\cQ \mid \hat f\cQ)$, as claimed.
\end{proof}

It follows from Lemma~\ref{l.Pesin.extension} and
Proposition~\ref{p.entropygenerating} that
\begin{equation}\label{eq.gPesin.formula}
H_{\hat\eta}( \hat f^{-1} \hat\cQ \mid \hat\cQ)
        = \int \log J_{\hat \nu} \hat f \, d\hat \eta.
\end{equation}
With this in mind we obtain the following

\begin{lemma}
$\hat\eta$ admits a disintegration $(\hat\eta_{\hat x})_{\hat x}$
along the measurable partition $\hat\cQ$ such that
\begin{equation}\label{eq.disintegrated}
\hat\eta_{\hat x} (B)
    = \frac{1}{Z(\hat x)} \int_{\hat\cQ(\hat x) \cap B} \De(\hat x,\hat y)
    \;d\hat\nu_{\hat x}(\hat y),
    \quad\text{where}\quad
    Z(\hat x)=\int_{\hat\cQ(\hat x)} \De(\hat x,\hat y) \,d\hat\nu_{\hat
x}(\hat y)
\end{equation}
for every measurable set $B$ and $\hat\eta$-almost every $\hat x$.
In consequence $\hat\eta_{\hat x}$ is absolutely continuous with
respect to $\hat\nu_{\hat x}$ for almost every $\hat x$.
\end{lemma}

\begin{proof}
Recall that $\De(\hat x ,\hat y)$ is well defined for almost every
$\hat x$ and every $\hat y \in \hat\cQ(\hat x)$ according to
Lemma~\ref{l.Pesin.extension}. In particular
Corollary~\ref{c.positivemeasure} implies that $0<Z(\hat x)<\infty$
almost everywhere.
Let $\rho_{{\hat x}}$ denote the measure in the right hand side of
the first equality in \eqref{eq.disintegrated}. Since $\hat
f^{-1}\hat \cQ\succ \hat\cQ$ a simple computation involving a change
of coordinates gives that
$$
\rho_{\hat x}((\hat f^{-1}\hat\cQ)({\hat x}))
    = \frac{1}{Z({\hat x})} \int_{(\hat f^{-1}\hat\cQ)({\hat x})} \De({\hat x},{\hat y}) \,d\hat\nu_{\hat x}({\hat y})
    = \frac{Z(\hat f({\hat x}))}{Z({\hat x}) \;J_{\hat\nu} \hat f({\hat x})}.
$$
We claim that
$$
-\int \log \rho_{\hat x}((\hat f^{-1}\hat\cQ)({\hat x})) \,d\hat\eta
        = \int \log J_{\hat\nu} \hat f \,d\hat\eta.
$$
Since $\rho_{\hat x}$ is a probability measure then $-\log
\rho_{\hat x}((\hat f^{-1}\hat\cQ)({\hat x}))$ is a positive
function and clearly the negative part of this function belongs to
$L^1(\hat\eta)$. Using that $J_\nu f$ is bounded away from zero and
infinity the same is obviously true also for $\log\frac{Z(\hat
f({\hat x}))}{Z({\hat x})}$. So, Birkhoff's ergodic theorem yields
that the limit
$$
 \omega(\hat x):= \lim_{n \to \infty} \frac1n \log Z(\hat f^n(\hat x))
        =\lim_{n \to \infty} \frac1n \log \frac{Z(\hat f^n(\hat x))}{Z(\hat x)}
        =\lim_{n \to \infty} \frac1n \sum_{j=0}^{n-1} \log \frac{Z \circ \hat f (\hat f^j(\hat x))}{Z(\hat f^j(\hat x))}
$$
do exist (although possibly infinite) and that
$$
\int \omega(\hat x) d\hat\eta(\hat x)
    = \int \log\frac{Z(\hat f({\hat x}))}{Z({\hat x})} d\hat\eta(\hat x).
$$
Since $Z$ is almost everywhere positive and finite, the sequence
$1/n \log Z(\hat f^n(\hat x))$ converge to zero in probability and,
consequently, it is almost everywhere convergent to zero along some
subsequence $(n_j)_j$. This shows that $\omega(\hat x)=0$ for
$\hat\eta$-almost every $\hat x$ and proves our claim.
On the other hand using relation \eqref{eq.gPesin.formula} and the
equality
$$
H_{\hat\eta}(\hat f^{-1}\hat\cQ \mid \hat\cQ)
    = - \int \log \hat\eta_{{\hat x}}(\hat f^{-1}\hat\cQ({\hat x}))
    \;d\hat\eta({\hat x})
$$
we obtain
$$
\int \log  \Big( \frac{d\hat\rho}
                {d \hat\eta} \Big|_{\hat f^{-1} \hat\cQ}
            \Big) d\hat\eta=0.
$$
Since the logarithm is a strictly concave function then
$$
0=\int \log  \Big( \frac{d\hat\rho_{\hat x}}
                {d \hat\eta_{\hat x}} \Big|_{\hat f^{-1} \hat\cQ} \Big)
                d\hat\eta
        \leq \log \Big( \int \frac{d\hat\rho_{\hat x}}{d \hat\eta_{\hat x}} \Big|_{\hat f^{-1} \hat\cQ}
                d\hat\eta\Big)=0,
$$
and the equality holds if and only if the Radon-Nykodym derivative
$\frac{d\hat\rho_{\hat x}}{d \hat\eta_{\hat x}}$ restricted to the
sigma-algebra generated by ${\hat f^{-1} \hat\cQ}$ is almost
everywhere constant and equal to one. Replacing $\hat f$ by any
power $\hat f^n$ in the previous computations it is not difficult to
check that $\hat\eta_{\hat x}$ and $\hat\rho_{\hat x}$ coincide in
the increasing family of sigma-algebras generated by the partitions
$\hat f^{-n}(\hat\cQ)$, $n \geq 1$.
Proposition~\ref{p.generating.partition}\eqref{ep4} readily implies
that $\hat\eta_{\hat x} = \rho_{\hat x}$ at $\hat\eta$-almost every
$\hat x$, which completes the proof of the lemma.
\end{proof}

We know from the previous lemma that $\hat\eta_{\hat x} \ll
\hat\nu_{\hat x}$ almost everywhere. Then, using that $\Wloc^u(\hat
x)$ is a neighborhood of $x_0$ in $M$ and the bijection
$$
\pi\mid_{\hWloc^u(\hat x)} : \hWloc^u(\hat x) \to \Wloc^u(\hat x)
$$
it follows that $\pi_*\hat\eta_{\hat x} \ll \nu$ for
$\hat\eta$-almost every $\hat x$. Since $(\hat\eta_{\hat x})$ is a
disintegration of $\hat\eta$ and $\pi_*\hat\eta=\eta$ it is
immediate that $\eta\ll\nu$. This completes proof of the theorem.
\end{proof}

\begin{remark}\label{rmk.general.endomorphism}
We point out there is an analogous version of
Theorem~\ref{thm.equilibrium is acim} that holds for piecewise
differentiable maps $f$ that behave like a power of the distance to
a possible critical or singular locus, as considered in
\cite{ABV00}. Indeed, assume that $\phi$ is an H\"older continuous
potential and $\nu$ is an expanding conformal measure such that
$J_\nu f=\la e^{-\phi}$ is H\"older continuous, where
$\la=\exp{\Ptop(f,\phi)}$. Assume also that $\eta$ is an equilibrium
state for $f$ with respect to $\phi$ and $\eta(\supp\nu)=1$. If
$\eta$ has non-uniform expansion and satisfies a slow recurrence
condition then there is a local unstable leaf passing through almost
every point, in the same way as in
Proposition~\ref{p.Pesinpointwise}. The construction of an
increasing partition as in Proposition~\ref{p.generating.partition}
and the proof of the absolute continuity of $\eta$ with respect to
$\nu$ remains unaltered. This is of independent interest and can be
applied, e.g. when $f$ is a quadratic map with positive Lyapunov
exponent, $\phi=-\log|\det Df|$ and $\nu$ is the Lebesgue measure to
prove the uniqueness of the SRB measure.
\end{remark}

Through the remaining of the section assume that $f$ is
topologically mixing. Since equilibrium states coincide with the
invariant measures that are absolutely continuous with respect to
$\nu$ then there is only one equilibrium state $\mu$ for $f$ with
respect to $\phi$. Thus, Theorem~\ref{Thm. Equilibrium States2} is a
direct consequence of Proposition~\ref{prop. acim} and the previous
statement. To finish the proof of Theorem~\ref{Thm. Equilibrium
States} it remains only to show exactness of the equilibrium state:

\begin{lemma}
$\mu$ is exact.
\end{lemma}

\begin{proof}
Let $E \in \cB_\infty$ be such that $\mu(E)>0$ and let $\vep>0$ be
arbitrary. There are measurable sets $E_n \in \cB$ such that
$E=f^{-n}(E_n)$. On the other hand, since $\mu$ is regular there
exists a compact set $K$ and an open set $O$ such that $K \subset E
\cap H \subset O$ and $\mu(O \setminus K) < \vep \mu(K)$, where $H$
denotes as before the set of points with infinitely many hyperbolic
times and $\vep>0$ is small. The same argument used in the proof of
Lemma~\ref{support Gibbs} shows that there exists $\tau>0$ $n \geq
1$ and $x \in H_n$ such that
$$
\frac{\mu(B(x,n,\de/4)\setminus E)}{\mu(B(x,n,\de/4))}<
\tau^{-1}\vep.
$$
Since $n$ is a hyperbolic time then $f^n \mid_{B(x,n,\de)}$ is a
homeomorphism that satisfies the bounded distortion property. Hence
$$
\frac{\mu(B(f^n(x),\de/4) \setminus f^n(E))}{\mu(B(f^n(x),\de/4))}
    <K_0 \tau^{-1}\vep.
$$
The topologically mixing assumption guarantees the existence of a
uniform $N \geq 1$ (depending only on $\de$) such that every ball of
radius $\de/4$ is mapped onto $M$ by $f^N$. Furthermore, since $\mu
\ll \nu$ with density $h=\frac{d\mu}{d\nu}$ bounded away from zero
and infinity then $J_{\mu} f = J_\nu f \; (h \circ f )/h$ satisfies
$C^{-1} \leq J_{\mu} f \leq C$ for some constant $C>1$. In
particular, since $d^N$ is an upper bound for the number of inverse
branches of $f^N$, $C$ bounds the maximal distortion of the Jacobian
at each iterate and $\mu$ is $f$-invariant we obtain that
$$
\mu(M \setminus E)
    = \mu(M \setminus E_{n+N})
    < K_0 d^N C^N \tau^{-1} \vep.
$$
The arbitrariness of $\vep>0$ shows that $\mu(E)=1$. This proves
that $\mu$ is exact.
\end{proof}

We finish this section with the

\begin{proof}[Proof of Corollary~\ref{Cor.ContinuousPotentials}]
If $\phi$ is a continuous potential satisfying (P), the existence of
an equilibrium state for $f$ with respect to $\phi$ will follow from
upper semi-continuity of the metric entropy. Let $\{\phi_n\}$ be a
sequence of H\"older continuous potentials satisfying (P) and
converging to $\phi$ in the uniform topology. Take $\mu_n$ to be an
equilibrium state for $f$ with respect to $\phi_n$, given by
Theorem~\ref{Thm. Equilibrium States2}, and let $\mu$ be an
accumulation point of the sequence $(\mu_n)_n$. Note that the
constants $c$ and $\de$ given by Lemma~\ref{delta} are uniform for
every $\mu_n$. So, any partition $\cR$ of diameter smaller than
$\delta$ that satisfies $\mu(\partial\cR)=0$ is generating with
respect to $\mu_n$, and
$$
h_{\mu}(f,\cR) \geq \limsup h_{\mu_n}(f,\cR).
$$
Using the continuity of $\phi \mapsto \Ptop(f,\phi)$ and $\phi
\mapsto \int \phi\; d\mu$ it follows that
\begin{multline*}
h_{\mu}(f, \cR)
    = \limsup_{n\to\infty}
        \Big[\Ptop(f,\phi_n) - \int \phi_n\; d\mu_n \Big]
    = \Ptop(f,\phi) - \int \phi\; d\mu
    \geq h_\mu(f).
\end{multline*}
This proves that $\mu$ is an equilibrium state for $f$ with respect
to $\phi$. Furthermore, the function
$$
(\eta,\phi) \mapsto h_\eta(f)+\int \phi \,d\eta
$$
is upper-semicontinuous on the product space of $c$-expanding
measures and convex set of continuous potentials satisfying (P).
Hence, proceeding as in \cite[Corollary~9.15.1]{Wa82} there exists a
residual $\mathcal R \subset C(M)$ of potentials satisfying (P) such
that there is a unique equilibrium state for $f$ with respect to
$\phi$. The proof of the corollary is now complete.
\end{proof}

\section{Stability of equilibrium states}\label{Proof of Theorems 3, 4 and 5}

\subsection{Statistical stability}\label{subsec.statistical}

Here we prove upper semi-continuity of the metric entropy and use
the continuity assumption on the topological pressure to prove that
the equilibrium states vary continuously with respect to the data
$f$ and $\phi$.

\begin{proof}[Proof of Theorem~\ref{Thm. Statistical
Stability}]
Let $\mathcal W$ be the set of H\"older continuous potentials and
$\cF$ the set of local homeomorphisms introduced in
Subsection~\ref{stability eq. states}. The strategy is to construct
a generating partition for \emph{all} maps in $\cF$. A similar
argument was considered in \cite{Ar07a}. Fix $(f,\phi) \in \cF
\times \cW$ and arbitrary sequences $\cF \ni f_n \to f$ in the
uniform topology, with $L_n \to L$ in the uniform topology, and $\cW
\ni \phi_n \to \phi$ in the uniform topology, let $\mu_n$ be an
equilibrium state for $f_n$ with respect to $\phi_n$ and $\eta$ be
an $f$-invariant measure obtained as an accumulation point of the
sequence $(\mu_n)_n$.

We begin with the following observation. Since the constants $c$ and
$\de$ given by Lemma~\ref{delta} are uniform in $\cF$, any partition
$\cR$ of diameter smaller than $\de/2$ satisfying $\eta(\partial
\cR)=0$ generates the Borel sigma-algebra for every $g \in \cF$.
Then, Kolmogorov-Sinai theorem implies that
$h_{\mu_n}(f_n)=h_{\mu_n}(f_n,\cR)$ and
$h_{\eta}(f)=h_{\eta}(f,\cR)$, that is,
$$
h_{\mu_n}(f_n)
 = \inf_{k \geq 1} \frac{1}{k} H_{\mu_n}(\cR_n^{(k)})
 \quad \text{and} \quad
 h_{\eta}(f) 
  = \inf_{k \geq 1} \frac{1}{k} H_{\eta}(\cR^{(k)}),
$$
where $H_{\eta}(\cR)=\sum_{R \in \cR} -\eta(R) \log \eta(R)$ and we
considered the dynamically defined partitions
$$
\cR_n^{(k)}=\bigvee_{j=0}^{k-1} f_n^{-j} (\cR)
    \quad \text{and} \quad
\cR^{(k)}=\bigvee_{j=0}^{k-1} f^{-j} (\cR).
$$
Since $\eta$ gives zero measure to the boundary of $\cR$ then
$H_{\mu_n}(\cR_n^{(k)})$ converge to $H_{\eta}(\cR^{(k)})$ as $n\to
\infty$ by weak$^*$ convergence. Furthermore, for every $\vep>0$
there is $N \geq 1$ such that
$$
h_{\mu_n}(f_n)
  \leq \frac{1}{N} H_{\mu_n}(\cR_n^{(N)})
  \leq \frac{1}{N} H_{\eta}(\cR^{(N)})+\vep
  \leq h_\eta(f)+2\vep.
$$
Recalling the continuity assumption of the topological pressure
$\Ptop(f,\phi)$ on the data $(f,\phi)$, that $\mu_n$ is an
equilibrium state for $f_n$with respect to $\phi_n$, and that $\int
\phi_n d\mu_n \to \int \phi \,d\eta$ as $n \to \infty$, it follows
that
$$
h_\eta(f)+\int \phi \,d\eta \geq \Ptop(f,\phi).
$$
This shows that $\eta$ is an equilibrium state for $f$ with respect
to $\phi$. Since every equilibrium state belongs to the convex hull
of ergodic equilibrium states and these coincide with finitely many
ergodic measures absolutely continuous with respect to $\nu$ (recall
Theorem~\ref{Thm. Equilibrium States2}), this completes the proof of
Theorem~\ref{Thm. Statistical Stability}.
\end{proof}

We finish this subsection with some comments on the assumption
involving the continuity of the topological pressure. The map $\phi
\mapsto P(f,\phi)$ varies continuously, provided that $f$ is a
continuous transformation (see for instance \cite[Theorem
9.5]{Wa82}). On the other hand, in this setting the topological
pressure $\Ptop(f,\phi)$ coincides with $\log \la_{f,\phi}$, where
$\la_{f,\phi}$ is the spectral radius of the transfer operator
$\cL_{f,\phi}$, \emph{for every $f\in \cF$ and every $\phi \in
\cW$}. Moreover, the operators $\cL_{f,\phi}$ vary continuously with
the data $(f,\phi)$. So, the continuous variation of the topological
pressure should be a consequence of the most likely spectral gap for
the transfer operator $\cL_{f,\phi}$ in the space of H\"older
continuous observables. Such a spectral gap property was obtained by
Arbieto, Matheus~\cite{AM} in a related context.

\subsection{Stochastic stability}

The results in this section are inspired by some analogous in
\cite{AA03}. First we introduce some definitions and notations.
Given $\un f \in \cF^\N$, define $\un f ^j=f_j \circ \dots f_2 \circ
f_1$. Let $(\theta_\vep)_{0<\vep\leq 1}$ be a family of probability
measures in $\cF$. Given a (not necessarily invariant) probability
measure $\nu$, we say that $(f,\nu)$ is \emph{non-uniformly
expanding along random orbits} if there exists $c>0$ such that
$$
\limsup_{n \to \infty} \frac{1}{n} \sum_{j=1}^{n} \log \|Df(\un
f^j(x))^{-1}\| \leq -2c<0
$$
for $(\theta_\vep^\N\times \nu)$-almost every $(\un f,x)\in \cF^\N
\times M$. If this is the case, Pliss's lemma guarantees the
existence of infinitely many hyperbolic times for almost every point
where, in this setting, $n\in\N$ is a \emph{$c$-hyperbolic time for
$({\un f},x)\in \cF^\N \times M$} if
$$
\prod_{j=n-k}^{n-1} \|Df(\un f^j(x))^{-1}\|< e^{-ck}
    \quad \text{for every $0\leq k\leq n-1$}.
$$
We refer the reader to \cite[Proposition 2.3]{AA03} for the proof.
Given $\vep>0$, let $n_1^{\vep}: \cF^\N \times M \to \N$ denote the
first hyperbolic time map. Set also $H_n({\un f})=\{x \in M : n \;
\text{is a $c$-hyperbolic time for}\; ({\un f},x)\}$. In the
remaining of the section let $f \in \cF$ and $\nu$ be an expanding
conformal measure such that $\supp\nu=H$. The next result shows that
$f$ has random non-uniform expansion. More precisely,

\begin{lemma}\label{lem. nue along random orbits}
Let $(\theta_\vep)_{0<\vep\leq 1}$ be a family of probability
measures in $\cF$ such that $\supp\theta_\vep$ is contained in a
small neighborhood $V_\vep(f)$ of $f$ and $\bigcap_{\vep}
V_\vep(f)=\{f\}$. If $\cF\ni g \mapsto J_\nu g$ is a continuous
function and $\vep$ is small enough then $(f,\nu)$ is non-uniformly
expanding along every random orbit of $(\hat f,\theta_\vep)$.
Furthermore,
$$
(\theta_\vep^\N \times \nu)
    (\left\{(\un f,x) \in \cF^\N \times M : n_1^\vep({\un f},x)>k \right\})
$$
decays exponentially fast and, consequently, $\int n_1^\vep
\,d\,(\theta_\vep^\N \times\nu) < \infty$.
\end{lemma}

\begin{proof}
Given $g \in \cF$, let $\cA_g \subset M$ be the region described in
(H1) and (H2). Denote by $\ti \cA$ the enlarged set obtained as the
union of the regions $\cA_g$ taken over all $g \in
\supp\theta_\vep$. If $\vep>0$ is small enough then we can assume
that $\ti \cA$ is contained in the same $q$ elements of the covering
$\cP$ as the set $\cA_f$.

Now we claim that, if $\ga$ is chosen as before and ${\un f} \in
\cF^\N$ the measure of the set
$$
B(n, {\un f})
  = \Big\{ x \in  M : \frac{1}{n} \#\{0\leq j \leq n-1: \un f^j(x) \in \ti\cA\} \geq \ga \Big\}
$$
decays exponentially fast. Indeed, the same proof of
Lemma~\ref{l.combinatorio} yields that $B(n, {\un f})$ is covered by
at most $e^{(\log q+\vep_0/2) n}$ elements of $\cP^{(n)}({\un
f})=\bigvee_{j=0}^{n-1} \un f^{-j}(\cP)$, for every large $n$. On
the other hand, since $\supp(\theta_\vep)$ is compact the function
$\supp\theta_\vep \ni g \mapsto J_\nu g$ is uniformly continuous:
for every $\vep>0$ there exists $a(\vep)>0$ (that tends to zero as
$\vep\to 0$) such that
$$
e^{-a(\vep)} \leq \frac{J_\nu f(x)}{J_{\nu} g(x)} \leq e^{a(\vep)}
$$
for every $g \in \supp(\theta_\vep)$ and every $x\in M$. As in the
proof of Proposition~\ref{p.recurrence}, this implies that
$$
1 \geq \nu(\un f^n(P))
    = \int_P \prod_{j=0}^{n-1} J_\nu f_j \circ \un f^j \, d\nu
    \geq e^{-a(\vep) n} \int_P J_\nu f^n \, d\nu
    > e^{(\log q + \vep_0 -a(\vep)) n} \; \nu(P)
$$
and, consequently, $\nu(P) \leq e^{-(\log q + \vep_0 -a(\vep)) n}$
for every $P \in \cP^{(n)}(\un f)$ and every large $n$. Hence
$$
\nu(B(n,{\un f}))
    \leq
    \#  \{P \in \cP^{(n)}(\mathrm{{\un f}}): P \cap B(n,\mathrm{{\un f}}) \neq \emptyset\}
    \times e^{-(\log q + \vep_0 -a(\vep)) n}
$$
which decays exponentially fast and proves the claim. Then, the set
$$
\un B(n)
    = \Big\{ ({\un f},x) \in \cF^\N \times M : \frac{1}{n}
    \#\{0\leq j \leq n-1: \un f^j(x) \in \ti\cA\} \geq \ga \Big\}
$$
is such that $(\theta_\vep \times \nu)(B(n))= \int \nu\left(
B(n,{\un f}) \right)\, d\theta_\vep^\N({\un f})$ also decays
exponentially fast. Borel-Cantelli guarantee that the frequency of
visits of the random orbit $\{\un f^j(x)\}$ to $\ti \cA$ is smaller
than $\ga$ for $\theta_\vep^\N\times\nu$-almost every $(\un f,x)$.
Moreover, since every $g \in \cF$ satisfy (H1) and (H2) with uniform
constants this proves that $f$ is non-uniformly expanding along
random orbits. Moreover, the first hyperbolic time map $n_1^\vep$ is
integrable because
$$
\int n_1 \, d(\theta_\vep^\N \times \nu)
   = \sum_{n\geq 0} (\theta_\vep^\N \times \nu)(\{n_1>n\})
   \leq  \sum_{n\geq 0} (\theta_\vep^\N \times \nu)(B(n))
   <\infty.
$$
This completes the proof of the lemma.
\end{proof}

\begin{remark}
Before proceeding with the proof, let us discuss briefly the
continuity assumption on $\cF \ni g \to J_\nu g$. First notice that
in our setting this is automatically satisfied when $\nu$ coincides
with the Lebesgue measure since it reduces to the continuity of $g
\mapsto \log |\det Dg|$. Given $g \in \cF$, let $\nu_g$ denote the
expanding conformal measure and set $P_g=\Ptop(f,\phi)$. Observe
that if $k$ is a $c$-hyperbolic time for $x$ with respect to $f$
then it is a $c/2$-hyperbolic time for $x$ with respect to every $g$
sufficiently close to $f$. Consequently
$$
K(c/2,\de)^{-2} e^{-|P_f-P_g| k}
   \leq \frac{\nu_g(B(x,k,\de))}{\nu_f(B(x,k,\de))}\leq
K(c/2,\de)^2 e^{|P_f-P_g| k},
$$
which proves that the conformal measures $\nu_f$ and $\nu_g$ are
comparable at hyperbolic times and that $J_\nu
g=d(g^{-1}_*\nu)/d\nu$ is a well defined object in the domain of
each inverse branch $g^{-1}$. So, in general, the relation above
indicates that the continuity of the topological pressure should
play a crucial role to obtain the continuity of the Jacobian $\cF
\ni g \to J_\nu g$.
\end{remark}

Given $n\geq 1$ define $f_x^n:\cF^\N \to M$ given by $f_x^n({\un
g}):=\un g^n(x)$. Since $f$ is non-uniformly expanding and
non-uniformly expanding along random orbits then there are finitely
many ergodic stationary measures absolutely continuous with respect
to $\nu$. More precisely,

\begin{theorem}\label{Thm. Stochastic Stability}
Let $(\theta_\vep)_\vep$ be a non-degenerate random perturbation of
$f \in \cF$. Given $\vep>0$ there are finitely many ergodic
stationary measures $\mu^\vep_1, \mu^\vep_2, \dots , \mu^\vep_l$
that are absolutely continuous with respect to the conformal measure
$\nu$ and
\begin{equation}\label{eq. physical stationary}
\mu^\vep_i
    = \lim_{n \to \infty}
    \frac{1}{n} \sum_{j=0}^{n-1} \int \un f^j_*  (\nu | B(\mu^\vep_i)) \; d\theta_\vep^\N(\un{f}),
\end{equation}
for every $1 \leq i \leq l$. In addition, $l\geq 1$ can be taken
constant for every sufficiently small $\vep$.
\end{theorem}

\begin{proof}
This proof follows closely the one of Theorem~C in \cite{AA03}. For
that reason we give a brief sketch of the proof and refer the reader
to \cite{AA03} for details. It is easy to check that any
accumulation point $\mu^\vep$ of the sequence of probability
measures
\begin{equation}\label{eq. conv. stationary}
\frac{1}{n} \sum_{j=0}^{n-1} (f_x^j)_* \theta_\vep^\N
\end{equation}
on $M$ is a stationary measure. Moreover, any stationary measure
$\mu^\vep$ is absolutely continuous with respect to $\nu$ because of
the non-degeneracy of the random perturbation and
$$
\mu^\vep(E)
  =\int \mu^\vep(g^{-1}(E)) \; d\theta_\vep(g)
  = \int 1_E(g(x)) \; d\theta_\vep(g)\, d\mu^\vep(x)
  =\int ((f_x)_*\theta_\vep^\N) (E) \;d\mu^\vep
$$
for every measurable set $E$.

On the other hand, by the ergodic decomposition of the $F$-invariant
probability measure $\theta_\vep^\N \times \mu^\vep$ there are
ergodic stationary measures. We prove that there can be at most
finitely many of them. Indeed, a point $x$ belongs to the basin of
attraction $B(\mu^\vep)$ of an ergodic stationary measure $\mu^\vep$
if and only if
\begin{equation}\label{eq.basin}
\frac{1}{n} \sum_{j=0}^{n-1} \psi
  (\un f^j(x)) \to \int \psi \,d\mu^\vep
\end{equation}
for every $\psi \in C(M)$ and $\theta_\vep^\N$-almost every $\un f
\in \cF^\N$. In addition, if $x\in B(\mu^\vep)$ then $g(x) \in
B(\mu^\vep)$ for every $g \in \supp(\theta_\vep)$. Furthermore, the
non-degeneracy of the random perturbation implies that $B(\mu^\vep)$
contains the ball of radius $r_\vep$ centered at $f(x)$. Then, the
compactness of $M$ implies that there are finitely many ergodic
absolutely continuous stationary measures $\mu_1^\vep, \dots
,\mu_l^\vep$, with $1 \leq l \leq l(\vep)$. Since
$\nu(B(\mu_i^\vep))>0$, integrating \eqref{eq.basin} with respect to
$\nu$ and using the dominated convergence theorem one obtains
$$
\int \psi \,d\mu_i^\vep
  =\lim_{n} \frac{1}{n} \sum_{j=0}^{n-1}
          \int_{B(\mu^\vep_i)} \psi \circ \un f^j \;d\nu
  = \lim_{n} \frac{1}{n} \sum_{j=0}^{n-1}
          \int \psi \,d \,\un f^j_*(\nu| B(\mu^\vep_i))
$$
for every $\psi \in C(M)$ and $\theta_\vep^\N$-almost every $\un f
\in \cF$. This proves the first statement of the theorem.

It remains to show that $l=l(\vep)$ can be chosen constant for every
sufficiently small $\vep$. The support of each stationary measure
$\mu_i^\vep$ is an invariant set with non-empty interior (see
\cite{AA03}). Since $f$ is non-uniformly expanding then
$\supp(\mu_i^\vep)$ contains some hyperbolic pre-ball $V_n(x)$
associated to $f$ and, by invariance, a ball of radius $\de$. This
proves that $l(\vep)\leq l_0$ for every small $\vep>0$. On the other
direction, since the set $\supp(\mu_i^\vep)$ has positive
$\nu$-measure and is forward invariant by $f$ it must be contained
in the support of some ergodic stationary measure $\mu_i^{\vep'}$
for every $\vep'$ smaller than $\vep$. This proves the $l$ can be
taken constant for small $\vep$ and completes sketch of the proof of
the theorem.
\end{proof}

Now we are in a position to prove that the equilibrium states
constructed in Theorem~\ref{Thm. Equilibrium States} are
stochastically stable.

\begin{proof}[Proof of Theorem~\ref{Thm. Stochastic Stability2}]
Let $(\mu^\vep)_{\vep>0}$ be a sequence of stationary measures
absolutely continuous with respect to $\nu$ and let $\eta$ be any
weak$^*$ accumulation point. Theorem~\ref{Thm. Stochastic Stability}
implies that there is $l\geq 1$ such that there are exactly $l$
ergodic stationary measures $\mu^\vep_1, \dots, \mu^\vep_l$ that are
absolutely continuous with respect to $\nu$, for every sufficiently
small $\vep$. Furthermore,
\[
\mu_i^\vep = \lim_{n \to \infty} \nu^\vep_{n,i}
    \quad\text{where}\quad
\nu^\vep_{n,i}
    =\frac{1}{n} \sum_{j=0}^{n-1} \int \un f^j_* (\nu \mid
    B(\mu_ i^\vep)) \, d\theta_\vep^\N(\un f).
\]
Proceed as in the beginning of Subsection~\ref{sec. upper bound and
finitude} and write $\nu^\vep_{n} \leq \xi^\vep_{n} + \frac{1}{n}
\sum_{j=0}^{n-1} \eta^\vep_j$ with
$$
\xi^\vep_{n,i}
    = \frac{1}{n} \sum_{j=0}^{n-1} \int_{B(\mu_ i^\vep)}
    \un f^j_* (\nu \mid H_j({\un f})) \, d\theta_\vep^\N(\un f)
$$
and
$$
\eta^\vep_{n,j}
   =\sum_{k>0} \int_{B(\mu_i^\vep)} \un f^k_* \,\Big( [\un f^j_*(\nu
    \mid H_j({\un f}))]  \mid\{n_1^\vep(\cdot, \si^j({\un f}))>k\} \Big)\,
d\theta_\vep^\N({\un f}).
$$
The arguments from Section~\ref{Absolutely Continuous Measures} and
the uniform integrability of $\vep \mapsto n_1^\vep \in
L^{1}(\theta_\vep^\N \times \nu)$ yield that each measure
$\nu_{n,i}^\vep$ is absolutely continuous with respect to $\nu$ with
density bounded from above by a constant depending only on $\vep$.
By weak$^*$ convergence it follows that $\eta$ is also absolutely
continuous with respect to $\nu$ and, consequently, $\eta$ belongs
to the convex hull of finitely many ergodic equilibrium states
$\mu_1, \dots, \mu_k$ for $f$ with respect to $\phi$. This completes
the proof of the theorem.
\end{proof}


\end{document}